\documentclass[10pt,journal]{IEEEtran}

\usepackage{amsmath,amssymb,amsfonts}
\usepackage{graphicx}
\usepackage{enumerate}
\usepackage{subfigure}
\usepackage{enumerate}
\usepackage{xcolor}
\usepackage{balance}

\newcommand{\RR}{\mathbb{R}}

\newcommand{\zero}{\mathbf{0}}
\newcommand{\GG}{\mathcal{G}}
\newcommand{\LL}{\mathcal{L}}
\newcommand{\HH}{\mathcal{H}}

\newenvironment{Proof}{\noindent{\em Proof:\/}}{\hfill $\Box$\par}
\newtheorem{Theorem}{Theorem}
\newtheorem{Lemma}{Lemma}

\newtheorem{Assumption}{Assumption}

\newtheorem{Remark}{Remark}
\newtheorem{Corollary}{Corollary}

\newtheorem{Problem}{Problem}

\allowdisplaybreaks

\begin{document}
 \title{\bf  Leader-following  Consensus
 over Jointly Connected Switching Networks
 is Achievable for Exponentially  Unstable Linear  Systems
 }

\author{Yuhan~Chen,~Tao~Liu,~and~Jie~Huang,~\IEEEmembership{Life Fellow, IEEE}
\thanks{Y. Chen and T. Liu  are with the School of Automation and Intelligent Manufacturing,
Southern University of Science and Technology, Shenzhen 518055, China.
E-mail: {\tt  chenyh2023@mail.sustech.edu.cn; liut6@sustech.edu.cn}. (Corresponding author: Tao Liu)}
\thanks{J. Huang is with the Department of Mechanical and Automation Engineering,
The Chinese University of Hong Kong, Shatin, N.T., Hong Kong.
E-mail: {\tt jhuang@mae.cuhk.edu.hk}.}}
\maketitle

\begin{abstract}
The leader-following consensus problem for general linear multi-agent systems over
jointly connected switching networks has been a challenging problem and the solvability of the problem has been limited to the class of linear multi-agent systems 
 whose system matrix is marginally stable. 
 This condition is restrictive since it even excludes 
 the most commonly used double-integrator system. 
This paper presents a  breakthrough
by demonstrating that leader-following exponential consensus
is achievable for general linear multi-agent systems
over jointly connected switching networks,
even when the system matrix is exponentially unstable.
The degree of instability can be explicitly characterized by two key quantities
that arise from the jointly connected condition on a switching graph.
By exploiting duality,
we further show that the output-based distributed observer design problem
for a general leader system is solvable over jointly connected switching networks,
even when the system matrix is exponentially unstable.
This is also in sharp contrast to the existing distributed observers, which rely on
the assumption that the leader system is marginally stable.
\end{abstract}


\section{Introduction}

Consensus is arguably the most fundamental problem for multi-agent control systems.
It plays a pivotal role in addressing nearly all cooperative control problems,
such as formation control \cite{Lin2005}, \cite{LinWang2014},
distributed optimization \cite{Nedic09},  \cite{Nedic15},
cooperative output regulation \cite{SuHuang12},  \cite{SuHuang12Cyber},
distributed state estimation \cite{WangMorse18}, \cite{ZhangLu24},
and Nash equilibrium seeking of non-cooperative games over networks  \cite{YeHu17},  \cite{YeHu18}.
The core objective of consensus is to design a distributed control law for a multi-agent system
to ensure that all agents' states synchronize to a common trajectory.
Depending on the nature of the common trajectory, the 
 consensus problem can be further categorized into
the leaderless consensus problem and the leader-following consensus problem.
In the leaderless consensus problem, 
one does not concern the time profile of the agents' states and 
the common trajectory is collectively determined by the initial states of all agents. 
In contrast, in the leader-following consensus problem, 
one further requires that the common trajectory be generated 
by the state of an autonomous system called the leader system.
In this sense, the leader-following consensus problem 
somehow poses a more challenging problem than the leaderless consensus problem.

Early research on both the leaderless consensus problem and the leader-following consensus problem focused on multi-agent systems with simple agent dynamics.  
In particular, the two consensus problems were first studied for single integrators in, e.g., \cite{Jadbabaie2003}, \cite{Olfati-SaberMurray04}, and \cite{Ren2005},
and were later extended to double integrators in, e.g., \cite{Hu2007},  \cite{Ren2007}, and \cite{Ren2008}.
It is noted that the system matrix of a single-integrator multi-agent system
has a simple eigenvalue at the origin
and the system matrix of a double-integrator multi-agent system has
a repeated (non-simple) eigenvalue at the origin.
Subsequently, reference \cite{Tuna08} first 
generalized the leaderless consensus result while reference \cite{Ni2010} first 
extended the leader-following consensus result,
both for general linear multi-agent systems without restrictions on the eigenvalues of the system matrix.
Nevertheless, the findings in \cite{Hu2007}, \cite{Ni2010},  \cite{Ren2007}, \cite{Ren2008}, and \cite{Tuna08} all rely on the assumption that the communication network for the multi-agent system remains static and connected.

In practice, the communication network is often time-varying
and may become disconnected due to unexpected link failures or abrupt changes in the ambient environment.
This makes the study of consensus problems over \emph{jointly connected switching networks}---where connectivity is not required at every instant---particularly interesting and challenging.
The concept of joint connectivity was first introduced in \cite{Jadbabaie2003}
 to handle the leaderless and the leader-following consensus problems for a single-integrator multi-agent system.
 Since then, extensive efforts have been devoted to addressing
 both the leaderless consensus problem
\cite{MaQin21, Qin2014, SuHuang-2012TAC, SuLee22, WangZhuFeng2019TAC},
and the leader-following consensus problem \cite{Cheng2008, Hong2007, Ni2010, SuHuang-2012TAC, SuLee22}, for more general linear multi-agent systems over jointly connected switching networks.
Specifically, reference \cite{Hong2007} studied the leader-following consensus problem
for a group of double integrators subject to a static leader,
and reference \cite{Cheng2008} further extended the result to the case of an active leader
but with a constant velocity known to all followers.
By imposing two linear matrix inequalities
on the system matrix and the input matrix of a linear multi-agent system,
reference \cite{Ni2010} also considered the leader-following consensus problem over jointly connected switching networks.
One of the two linear matrix inequalities in \cite{Ni2010} was removed
by reference \cite{SuHuang-2012TAC}, thus only requiring that the system matrix be marginally stable. 
Moreover, both the leaderless and the leader-following asymptotic consensus problems over jointly connected switching networks were tackled in \cite{SuHuang-2012TAC}.
Reference \cite{Qin2014} considered the leaderless consensus problem for a linear multi-agent system
where the system matrix can contain exponentially unstable yet weak enough modes.
However, leaderless exponential consensus in \cite{Qin2014}
is obtained under the condition that the input matrix has a full row rank;
in other words, full-state couplings exist among the agents.
Reference \cite{MaQin21} explored necessary and sufficient conditions for
ensuring leaderless exponential consensus for neutrally stable linear multi-agent systems.
More recently, reference \cite{SuLee22} solved both the leaderless and the leader-following
exponential consensus problems for marginally stable linear multi-agent system via
distributed dynamic output feedback control.

Up to 2019, the marginal stability assumption had been a barrier to the advancement of the study of the two consensus problems over jointly connected switching networks. 
This assumption is undesirable because, under this assumption, 
one cannot even handle the double-integrator system 
which is the most commonly used model for unmanned aerial vehicles, 
mobile robots, and so on. This barrier was broken in 2019 when reference \cite{WangZhuFeng2019TAC} 
showed that the leaderless exponential consensus problem over jointly connected switching networks is solvable 
for linear multi-agent systems with exponentially unstable modes.
By resorting to infinite matrix product,  reference \cite{WangZhuFeng2019TAC} characterized the allowable degree of instability 
by a so-called synchronizability exponent of the switching topology.
One naturally wonders if the breakthrough made in \cite{WangZhuFeng2019TAC} for
the leaderless consensus problem can also be made for the leader-following consensus problem. Indeed, in this paper, we further study the leader-following exponential consensus problem
for general linear multi-agent systems over jointly connected switching networks.
Specifically, 
we first derive a critical value $\delta$ from the switching topology.
This value $\delta$, together with another key quantity $T_\mathrm{c}$ of the jointly connected condition, enables us to establish a bound on the permissible instability of the system matrix
of the linear multi-agent system.
Then, by developing a series of technical lemmas,
we prove that leader-following exponential consensus is achievable even for linear multi-agent systems with exponentially unstable modes, provided the instability remains within the established bound.

Furthermore, we also address the output-based distributed observer design problem
as the dual problem of the leader-following consensus problem.
Leveraging this duality, we show that an exponentially convergent output-based distributed observer can be readily designed for an exponentially unstable linear leader system over jointly connected switching networks.
Similarly, the instability of the leader system should remain within this established permissible instability bound.
The solvability of these two dual problems
represents a significant advance beyond the state of the art in \cite{CaiSuHuang2022},
where both the leader-following exponential consensus problem and the exponentially convergent output-based distributed observer design problem
over jointly connected switching networks were only solved for marginally stable linear multi-agent systems.

The rest of this paper is organized as follows.
Section \ref{Section-Problem} formulates the leader-following exponential consensus problem.
Section \ref{Section-Preliminares} states the main result on the solvability of this problem
and Section \ref{Section-Main-results} provides a detailed proof.
Further, Section \ref{Section-observer} formulates and studies the dual problem of the leader-following consensus problem,
namely, the output-based distributed observer design problem.
Then, Section \ref{Section-Example} presents numerical examples to validate the theoretical results.
Finally, Section \ref{Section-Conclusions} concludes the paper.

\medskip
\noindent\textit{Notation}: $\mathbf{1}_{N}$ denotes an $N$-dimensional column vector whose entries are all $1$.
$\zero $ denotes a matrix of zeros with appropriate dimensions.
Given $a_{i} \in \RR,\, i = 1, \dots,N$, $\mathrm{diag}\left\{a_{1},\ldots,a_{N}\right\}$ 
represents the diagonal matrix whose diagonal entries are $a_{1},\ldots,a_{N}$.
Given $x_{i}\in \RR^{n_{i}},\, i=1, \ldots,N$, $\mathrm{col}\left(x_{1},\ldots,x_{N}\right)=\begin{bmatrix} x_{1}^{T} & \cdots & x_{N}^{T} \end{bmatrix}^{T}$.
$\otimes$ denotes the Kronecker product of matrices.
The kernel of a matrix $A \in \RR^{m \times n}$ is
given by $\mathrm{ker}(A) = \left\{x \in \RR^n: Ax = 0 \right\}$.
$\left\|\cdot\right\|$ denotes the Euclidean norm or
the matrix norm induced by the Euclidean norm.
 $\prod_{r = m}^n \Xi_r$ denotes 
 the product of square matrices $\Xi_n\cdots \Xi_{m+1}\Xi_m$ when $n\ge m$
and $\prod_{r = m}^n \Xi_r$ is taken as an identity matrix when $n <m$.
We use $\lambda_{\max}(A)$ to denote the largest real part 
of the eigenvalues of a square matrix $A$,
i.e., $\lambda_{\max}(A) := \max \left\{\mathrm{Re}(\lambda) : \lambda \in \mathrm{spec}(A)\right\}$,
in which, $\mathrm{spec}(A)$ denotes the spectrum of the matrix $A$.
We call a time function
$\sigma : [0, \infty) \to \mathcal{P} = \left\{1, 2,\ldots, n_0\right\}$,
a piecewise constant switching signal,
if there exists a sequence $\left\{t_j : j =0,1,2,\ldots \right\}$ satisfying
$t_0 = 0$ and $t_{j+1} - t_{j} \geq \tau , \forall\, j =0,1,2,\ldots $, for some $\tau >0$ such that,
for all $t \in [t_j, t_{j+1})$, $\sigma(t) = p$, for some $p \in \mathcal{P}$.
Then $\mathcal{P}$ is called the switching index set,
$\left\{t_j : j =0,1,2,\ldots \right\}$ are called the switching instants,
and $\tau$ is called the dwell time.

\section{Problem Formulation}\label{Section-Problem}

Consider a general linear leader-follower multi-agent system consisting of one leader system
and $N$ follower subsystems.
The leader system is described as
\begin{align}\label{eq-leader}
	\dot x_0 (t) = A x_0 (t), \quad t \geq 0
\end{align}
where $x_0 (t) \in \RR^n$ is the state of the leader system,
and $A \in \RR^{n\times n}$ is the system matrix which can be unstable.
The follower system is described by
\begin{align}\label{eq-system}
	\dot x_i (t) = A x_i (t) + B u_i (t),\quad i = 1, \ldots, N, \quad t \geq 0
\end{align}
where $x_i (t) \in \RR^n$ and $u_i (t) \in \RR^m$ are the state and the control input of the $i$th subsystem, respectively,
and $B \in \RR^{n \times m}$ is the input matrix.

Define a switching graph\footnote{See Appendix \ref{Appendix-graph} for a summary of notation on graph.}
$\bar \GG_{\sigma (t)} = (\bar{\mathcal{V}}, \bar{\mathcal{E}}_{\sigma(t)})$
with the node set $\bar{\mathcal{V}} = \left\{0, 1, \ldots, N \right\}$
and the edge set $\bar{\mathcal{E}}_{\sigma(t)}$ dictated by a piecewise constant switching signal $\sigma(t)$.
We associate note $0$ with the leader system \eqref{eq-leader}
and node $i,\, i = 1, \ldots, N$, with the $i$th subsystem of the follower system \eqref{eq-system}.
Then the edge set $\bar{\mathcal{E}}_{\sigma(t)} \subseteq \bar{\mathcal{V}} \times \bar{\mathcal{V}}$
will be used to describe the communication constraints for
the leader-follower multi-agent system composed of \eqref{eq-leader} and \eqref{eq-system}.
Specifically, at any time instant $t$,
the control input $u_i (t)$ of the $i$th follower subsystem,
$i = 1, \ldots, N$, can access the state $x_j (t)$ of the $j$th agent,
$j = 0, 1, \ldots, N$, if and only if $(j, i) \in \bar{\mathcal{E}}_{\sigma(t)}$.
A collection of controllers designed to satisfy the communication constraints imposed by
 $\bar{\mathcal{E}}_{\sigma(t)}$ is called a distributed control law.

Subjecting to the above communication constraints,
we study a distributed static state feedback control law of the following form:
\begin{align}\label{eq-control-law}
	u_i(t) = K \sum_{j = 0}^{N} a_{ij} (t)
	\left( x_j (t) - x_i (t) \right), \quad i = 1, \ldots, N
\end{align}
where $K \in \RR^{m \times n}$ is the feedback gain matrix to be designed and $a_{ij} (t)$ are entries of the adjacency matrix $\bar{\mathcal{A}}_{\sigma(t)} := [a_{ij} (t)]_{i,j=0}^N \in \RR^{(N+1) \times (N+1)}$
of the switching graph $\bar \GG_{\sigma (t)}$.
For simplicity, we let $a_{ij}(t) = 1$ if $(j, i) \in \bar{\mathcal{E}}_{\sigma(t)}$ and $a_{ij}(t) = 0$ otherwise.

Now we describe the leader-following exponential consensus problem as follows.

\begin{Problem}\label{Problem}
Given the leader-follower multi-agent system composed of \eqref{eq-leader} and \eqref{eq-system},
and the switching graph $\bar \GG_{\sigma (t)}$,
find a distributed control law of the form \eqref{eq-control-law} such that,
for any initial conditions $x_i (0) \in \RR^n,\, i = 0, 1, \ldots, N$,
the solution of the closed-loop system exists over $[0, \infty)$ and satisfies
$\lim_{t \to \infty} (x_i (t) - x_0 (t)) = 0,\, i = 1, \ldots, N$, exponentially.
\end{Problem}

Let $\bar{\LL}_{\sigma(t)} \in \RR^{(N+1) \times (N+1)}$ denote the Laplacian of the switching graph $\bar{\GG}_{\sigma(t)}$
and partition $\bar{\LL}_{\sigma(t)}$ as follows:
\begin{align}\label{eq-bar-LL}
	\bar \LL_{\sigma (t)} = \left[\begin{array}{c|c}
	0 & \zero \\
	\hline -\Delta_{\sigma (t)} \mathbf{1}_N & \HH_{\sigma (t)} \end{array}\right]
\end{align}
where $\Delta_{\sigma (t)} = \mathrm{diag} \left\{a_{10}(t), \ldots , a_{N0}(t) \right\}$
and
\begin{equation}\label{eq-H-sigma-def}
  \HH_{\sigma (t)} = \LL_{\sigma (t)} + \Delta_{\sigma (t)}  \in \RR^{N\times N}.
\end{equation}
In particular, the switching matrix $\HH_{\sigma (t)}\in \RR^{N \times N}$ is called the leader-follower matrix
of the switching graph $\bar \GG_{\sigma (t)}$,
in which, $\LL_{\sigma (t)}$ is the Laplacian of the subgraph $\GG_{\sigma (t)}= (\mathcal{V}, \mathcal{E}_{\sigma(t)})$ of $\bar \GG_{\sigma (t)}$ with $\mathcal{V} = \left\{1, \ldots, N \right\}$ and $\mathcal{E}_{\sigma(t)} = (\mathcal{V} \times \mathcal{V}) \cap \bar{\mathcal{E}}_{\sigma(t)}$.

 Denote the consensus errors between  the follower system   \eqref{eq-system}
 and  the leader system \eqref{eq-leader} as
 \begin{equation}\label{eq-bar-x-i-def}
   \bar{x}_i (t) = x_i (t) - x_0 (t), \quad i = 1, \ldots, N
 \end{equation}
and let
\begin{equation}\label{eq-bar-x-def}
\bar{x}(t)  = \mathrm{col}\left(\bar{x}_1 (t), \ldots, \bar{x}_N (t)\right) \in \RR^{Nn}.
\end{equation}
Then, the error dynamics of the closed-loop system
composed of \eqref{eq-leader} to \eqref{eq-control-law} can be put into the following compact form \cite{CaiSuHuang2022}:
\begin{align}\label{eq-der-L0-otimes-x}
	\dot{\bar{x}}(t) = \left(I_{N}\otimes A - \mathcal{H}_{\sigma(t)}\otimes BK \right)\bar{x}(t), \quad t\ge 0.
\end{align}

From the definition of $\bar{x}(t)$ in \eqref{eq-bar-x-i-def} and \eqref{eq-bar-x-def}, Problem \ref{Problem} is solvable
if and only if $\lim_{t\to\infty}{\bar{x}}(t)=0$ exponentially.
Thus, the solvability of Problem~\ref{Problem} is converted to that of the exponential stabilization problem for the linear switched system \eqref{eq-der-L0-otimes-x}
via the design of the feedback gain matrix $K$.

Three standard assumptions needed for solving Problem~\ref{Problem} are listed below.

\begin{Assumption}\label{Ass-stabilizable}
	The matrix pair $\left(A, B\right)$ is controllable.
\end{Assumption}

\begin{Assumption}\label{Ass-undirected}
The subgraph $\GG_{\sigma (t)}$ is undirected for all $t\ge0$.
\end{Assumption}

\begin{Assumption}\label{Ass-connected}
	There exists a subsequence $\left\{t_{j_k}:k = 0, 1, 2, \ldots \right\}$ of the switching instants
    $\left\{t_j:j = 0, 1, 2, \ldots \right\}$ with $t_{j_0}=0$ and $t_{j_{k+1}} - t_{j_k} \leq T_\mathrm{c}$ for some $T_\mathrm{c} > 0$,
    such that every node $i, i = 1,\ldots , N$, is reachable from node 0 in the union graph
$\bigcup_{t_j\in[t_{j_k},t_{j_{k+1}})}\bar{\mathcal{G}}_{\sigma(t_j)}$.
\end{Assumption}

\begin{Remark}\label{Remark-graph}
 The switching graph ${\mathcal{G}}_{\sigma(t)}$ is called  jointly connected if 
there exists a subsequence $\left\{t_{j_k}:k = 0, 1, 2, \ldots \right\}$ of the switching instants
    $\left\{t_j:j = 0, 1, 2, \ldots \right\}$ with $t_{j_0}=0$ and $t_{j_{k+1}} - t_{j_k} \leq T_\mathrm{c}$ for some $T_\mathrm{c} > 0$,
    such that the union graph
$\bigcup_{t_j\in[t_{j_k},t_{j_{k+1}})}{\mathcal{G}}_{\sigma(t_j)}$ is connected \cite{Jadbabaie2003, SuHuang-2012TAC}. 
Consider the following linear switched system:
\begin{align*}
	\dot{\bar{x}}(t) = \left(I_{N}\otimes A - \mathcal{L}_{\sigma(t)}\otimes BK \right)\bar{x}(t), \quad t\ge 0
\end{align*}
which is the closed-loop error system arising from 
the leaderless consensus problem for linear multi-agent systems over switching networks.
It is noted that, under Assumptions~\ref{Ass-stabilizable} and~\ref{Ass-undirected} 
and the jointly connected condition of ${\mathcal{G}}_{\sigma(t)}$, 
solving the leaderless exponential consensus problem is equivalent to finding $K$ 
such that the null space of the matrix $\sum_{t_j\in[t_{j_k},t_{j_{k+1}})} \mathcal{L}_{\sigma(t_j)}\otimes I_n $ is exponentially stable \cite{SuHuang-2012TAC}.
\end{Remark}

\begin{Remark}\label{Remark-controllable}
 It is noted that if Problem \ref{Problem} is solvable,
 then Assumption \ref{Ass-stabilizable} can be  relaxed to 
 the case where the matrix pair $(A,B)$ is stabilizable.
 In this case, through a Kalman decomposition on the matrix pair $(A,B)$, we have
 \begin{equation*}
   TAT^{-1}=\begin{bmatrix}
             A_\mathrm{c}  & A_{12} \\
              \mathbf{0} &  A_\mathrm{u}
            \end{bmatrix}, \qquad TB=\begin{bmatrix}
                                      B_\mathrm{c} \\
                                      \mathbf{0}
                                    \end{bmatrix}
 \end{equation*}
 for some nonsingular matrix 
$$T = \begin{bmatrix}T_1 \\ T_2\end{bmatrix} \in \RR^{n \times n} \enspace \text{with} \enspace T_1 \in \RR^{n_c \times n} \enspace \text{and} \enspace T_2 \in \RR^{(n - n_c) \times n}, $$ 
such that 
 the matrix pair $(A_\mathrm{c}, B_\mathrm{c}) \in \RR^{n_\mathrm{c} \times n_\mathrm{c}} \times \RR^{n_\mathrm{c} \times m}$ is controllable
 and the matrix $A_\mathrm{u} \in \RR^{(n-n_\mathrm{c}) \times (n-n_\mathrm{c})}$ is Hurwitz
 with $0\le n_\mathrm{c} \le n$.
 Let the feedback gain matrix be
$K=\begin{bmatrix}
             K_{\mathrm{c}} & \mathbf{0} 
           \end{bmatrix} T$ with $K_{\mathrm{c}}\in \RR^{m \times n_{\mathrm{c}}}$. 
           Then, under the following coordinate transformation:
 \begin{equation*}
   \begin{bmatrix}
     \bar{x}_{\mathrm{c}}(t) \\
     \bar{x}_{\mathrm{u}}(t)
   \end{bmatrix} = \begin{bmatrix}I_{N} \otimes T_1\\I_{N} \otimes T_2\end{bmatrix}\bar{x}(t),
   \quad 
    \begin{array}{l}
        \bar{x}_{\mathrm{c}}(t)\in \RR^{N n_\mathrm{c}}\\
        \bar{x}_{\mathrm{u}}(t)\in \RR^{N (n-n_\mathrm{c})}
    \end{array}
 \end{equation*}
the linear switched system \eqref{eq-der-L0-otimes-x} becomes
\begin{subequations}\label{eq-bar-x-Kalman}
  \begin{align}
   \dot{\bar{x}}_{\mathrm{c}}(t) & = \left(I_{N}\otimes A_{\mathrm{c}} - \mathcal{H}_{\sigma(t)}\otimes B_{\mathrm{c}}K_{\mathrm{c}} \right)\bar{x}_{\mathrm{c}}(t) \notag \\
   &\quad\ +\left(I_{N} \otimes A_{12}\right)\bar{x}_{\mathrm{u}}(t)  \label{eq-bar-x_c}\\
   \dot{\bar{x}}_{\mathrm{u}}(t) &  = \left(I_{N}\otimes A_{\mathrm{u}}\right)\bar{x}_{\mathrm{u}}(t), \quad t\ge 0. \label{eq-bar-x_u}
  \end{align}
\end{subequations}
Since system \eqref{eq-bar-x_u} is exponentially stable,
system \eqref{eq-bar-x_c} is in the same form as system \eqref{eq-der-L0-otimes-x},
but subjecting to an exponentially decaying input.
Thus, if system \eqref{eq-der-L0-otimes-x} can be exponentially stabilized via the design of $K$,
then system \eqref{eq-bar-x-Kalman} can also be exponentially stabilized via a similar design of the matrix $K_{\mathrm{c}}$, and hence, Problem \ref{Problem} is solvable by $K=T\begin{bmatrix}
             K_{\mathrm{c}} & \mathbf{0} 
           \end{bmatrix}$.
\end{Remark}

\section{Statement of  Main Result}\label{Section-Preliminares}

Consider the leader-follower matrix $\HH_{\sigma(t)}$ given by \eqref{eq-H-sigma-def}.
Since the switching signal $\sigma(t)$ remains constant
over each time interval $[t_j, t_{j+1}),\, j =0,1,2,\ldots$,
we can denote the nullity of $\HH_{\sigma(t)}$
over each time interval $[t_j, t_{j+1})$ by a nonnegative integer $n_j$.
Furthermore, under Assumption \ref{Ass-undirected},
$\LL_{\sigma (t)}$ is symmetric and positive semi-definite, and so is $\HH_{\sigma (t)}$.
Thus, we can denote the eigenvalues of $\HH_{\sigma (t)}$ over each
time interval $[t_j, t_{j+1}),\, j =0,1,2,\ldots$, in an ascending order as follows:
\begin{equation*}
  0= \lambda_{1}^{(j)}=\cdots=\lambda_{n_{j}}^{(j)}<\lambda_{n_{j}+1}^{(j)}
\le \cdots  \le \lambda_{N}^{(j)}.
\end{equation*}
Meanwhile, we can find a corresponding set of unit eigenvectors as follows:
\begin{equation*}
  \left \{ \xi_1^{(j)}, \ldots, \xi_{n_j}^{(j)}, \xi_{n_j +1}^{(j)},\ldots, \xi_N^{(j)} \right\}
\end{equation*}
which forms an orthonormal basis of $\RR^{N}$.

Then, for $t\in [t_j, t_{j+1}),\, j=0,1,2,\ldots$, we let
\begin{subequations}\label{eq-orthonormal-basis}
  \begin{align}
  Q_{\sigma(t)}&= \begin{bmatrix}
                  \xi_1^{(j)}  & \cdots &  \xi_{n_j}^{(j)}
                 \end{bmatrix} \in \RR^{N \times n_{j}} \\
  \Gamma_{\sigma(t)} &=  \begin{bmatrix}
                           \xi_{n_j +1}^{(j)}  & \cdots &  \xi_N^{(j)}
                         \end{bmatrix}       \in \RR^{N\times (N-n_{j})}
\end{align}
\end{subequations}
and define
\begin{equation}\label{eq-P-sigma-def}
  P_{\sigma(t)} = Q_{\sigma(t)}Q_{\sigma(t)}^T \in \RR^{N \times N}.
\end{equation}
In particular, we let $P_{\sigma(t)} = \zero_{N \times N}$ when $n_j = 0$.
It is clear that $P_{\sigma(t)}$ is an orthogonal projection matrix
onto $\mathrm{ker}\left(\HH_{\sigma(t)}\right)$ at the time instant $t$
and $P_{\sigma(t)}$ satisfies
\begin{equation}\label{eq-P-sigma-P-norm}
  \left\| P_{\sigma(t)}\right\| \le 1, \quad \forall\, t\ge0.
\end{equation}

\begin{Remark}
It is interesting to point out that,
when the graph  $\bar \GG_{\sigma(t)}$ is static as $\bar \GG$,
Assumption \ref{Ass-undirected} reduces to that $\GG$ is undirected
and Assumption \ref{Ass-connected} reduces to that every node
$i,\, i = 1,\ldots , N$, is reachable from node 0 in the graph $\bar \GG$.
For this special case,
the leader-follower matrix $\HH_{\sigma(t)}$ becomes $\HH$, which is positive definite.
Consequently, the orthogonal projection matrix $P_{\sigma(t)}$
onto $\mathrm{ker}\left(\HH\right)$  becomes $\mathbf{0}_{N \times N}$.
\end{Remark}

\begin{Lemma}\label{lemma-connect-graph}
	Under Assumptions \ref{Ass-undirected} and \ref{Ass-connected},
	there exists a number $0 < \delta < 1$ such that
\begin{equation}\label{eq-norm-products-le-1-forall}
	\left\|\prod_{r=j_k}^{j_{k+1}-1} P_{\sigma(t_{r})}\right\| \leq \delta, \quad \forall\, k =0,1,2,\ldots.
\end{equation}
\end{Lemma}

\begin{Proof}
Fix a nonnegative integer $k$
and consider the following set of linear equations:
\begin{align}\label{eq-kerL}
	\left(I_{N} - P_{\sigma(t_r)}\right)x = 0,
\quad  r = j_{k}, j_{k}+1, \ldots, j_{k+1}-1.
\end{align}
Premultiplying each of the above equations by $\HH_{\sigma(t_{r})}$ yields
\begin{align*}
	\HH_{\sigma(t_r)} x = 0, \quad r = j_{k}, j_{k}+1, \ldots, j_{k+1}-1.
\end{align*}
Hence, we have $\left(\sum_{r=j_k}^{j_{k+1}-1} \HH_{\sigma(t_r)}\right)x=0$.
By \cite[Lemma~2.3]{CaiSuHuang2022}, the matrix $- \sum_{r=j_k}^{j_{k+1}-1} \HH_{\sigma(t_r)}$ is Hurwitz
if Assumption~\ref{Ass-connected} holds.
Thus, under Assumptions \ref{Ass-undirected} and \ref{Ass-connected},
the linear equations in \eqref{eq-kerL} have a unique solution $x = 0$, i.e.,
\begin{align*}
	\bigcap_{r=j_{k}}^{j_{k+1}-1} \mathrm{ker}\left(I_{N} - P_{\sigma(t_r)}\right)=\{0\}.
\end{align*}
Furthermore, since $\left\|P_{\sigma(t_{r})}\right\|  \le 1,\, \forall\, r=j_{k}, j_{k}+1, \ldots, j_{k+1}-1$,
by \cite[Lemma 4]{Jadbabaie2003}, we can obtain
\begin{equation}\label{eq-norm-products-le-1}
	\left\|\prod_{r=j_{k}}^{j_{k+1}-1} P_{\sigma(t_{r})}\right\| < 1.
\end{equation}

It is noted that, under Assumption \ref{Ass-connected},
$j_{k+1} - j_{k} \leq \left\lceil\frac{T_\mathrm{c}}{\tau} \right\rceil,\, \forall\, k=0,1,2,\ldots$.
Since the switching index set is finite,
the number of matrix products of the form $\prod_{r=j_{k}}^{j_{k+1}-1} P_{\sigma(t_{r})}$ is also finite.
Therefore, we can define
\begin{equation}\label{eq-delta-def}
  \delta:= \max_{k=0,1,2,\ldots}\left\|\prod_{r=j_{k}}^{j_{k+1}-1} P_{\sigma(t_{r})}\right\|.
\end{equation}
Then, it follows from \eqref{eq-norm-products-le-1} and \eqref{eq-delta-def}
that $0< \delta < 1$ and $\delta$ satisfies \eqref{eq-norm-products-le-1-forall}.
\end{Proof}

Now we are ready to state the main result of this paper.

\begin{Theorem}\label{Theorem}
  Under Assumptions \ref{Ass-stabilizable} to \ref{Ass-connected},
  Problem \ref{Problem} is solvable if the following condition holds:
	\begin{align}\label{eq-restriction-A}
		\lambda_{\max}(A) < - \frac{\ln \delta}{T_\mathrm{c}}
	\end{align}
 in which, $T_\mathrm{c} > 0$ and $0 < \delta < 1$ are asserted in Assumption \ref{Ass-connected} and Lemma \ref{lemma-connect-graph}, respectively.
\end{Theorem}

\begin{Remark}
The inequality \eqref{eq-restriction-A} indicates that
the leader-following exponential consensus problem for linear multi-agent systems
over jointly connected switching networks is solvable
even if the system matrix $A$ is exponentially unstable.
This is a significant advancement over the state of the art in \cite{CaiSuHuang2022},
which requires that the system matrix $A$ be marginally stable.
\end{Remark}

\begin{Remark}
The solvability condition \eqref{eq-restriction-A}
is in a similar form to the solvability condition (4) established in \cite{WangZhuFeng2019TAC} for the leaderless exponential consensus problem for general linear multi-agent systems over jointly connected switching networks.
The major difference is that our approach studies the products of orthogonal projection matrices onto the kernel of the leader-follower matrix $\mathcal{H}_{\sigma(t)}$
of the switching graph $\bar \GG_{\sigma (t)}$
to quantify the key value $\delta$. 
In contrast, 
reference \cite{WangZhuFeng2019TAC}
defines $P_{\sigma(t)}$ as an orthogonal projection matrix
onto $\mathrm{ker}\left(\LL_{\sigma(t)}\right)$,
i.e., the kernel of the Laplacian of the switching graph $\GG_{\sigma (t)}$ for the leaderless multi-agent system,
and then studies the products of the matrices $P_{\sigma(t_{r})}-P$ with $P=\frac{1}{N}\mathbf{1}_{N}\mathbf{1}_{N}^{T}$.
\end{Remark}


\section{Proof of Main Result}\label{Section-Main-results}

\subsection{Solution of Linear Switched System \eqref{eq-der-L0-otimes-x}}\label{Subsection-Problem-Conversion}

In this subsection, we  focus on analyzing the linear switched system \eqref{eq-der-L0-otimes-x}.
We first note that,
since the switching signal $\sigma(t)$ is piecewise constant,
 system \eqref{eq-der-L0-otimes-x} can be rewritten into the following form:
\begin{align}\label{eq-bar-x-rewrite}
	\dot{\bar{x}}(t) &= (I_{N}\otimes A -\mathcal{H}_{\sigma(t_j)}\otimes BK) \bar{x}(t), \quad \notag \\
 &\qquad\qquad\qquad\quad t \in [t_j, t_{j+1}),\  j =0,1,2,\ldots.
\end{align}

Next, over each time interval,
we perform an orthogonal transformation as follows:
\begin{align}\label{eq-x-hat}
	 \begin{bmatrix} \hat x_1(t) \\ \hat x_2(t) \end{bmatrix} & = \left(\begin{bmatrix}
	                                                              Q_{\sigma(t_j)}^{T} \\
	                                                              \Gamma_{\sigma(t_j)}^{T}
	                                                            \end{bmatrix} \otimes I_{n} \right)
\bar{x}(t),  \notag \\
& \qquad\qquad\qquad   t \in [t_j, t_{j+1}), \  j =0,1,2,\ldots
\end{align}
where $\hat x_1 (t)\in \RR^{n_j n}$, $\hat x_2 (t) \in \RR^{(N - n_j) n}$ with $n_j$ being the nullity of $\HH_{\sigma(t_j)}$,
and $Q_{\sigma(t_j)}$ and  $\Gamma_{\sigma(t_j)}$ are defined in \eqref{eq-orthonormal-basis}.

In the new coordinate, system \eqref{eq-bar-x-rewrite} reads as
\begin{subequations}\label{eq-x1x2}
\begin{align}
		\dot{\hat{x}}_{1}(t) & = \left(  I_{n_j}\otimes A \right) {\hat x_1}(t), \quad  t \in [t_j, t_{j+1}), \ j=0,1,2,\ldots \\
		\dot{\hat{x}}_{2}(t) & = \left(  I_{N-n_j} \otimes A - \Lambda_j \otimes B K \right) {\hat x_2}(t) \label{eq-x2}
\end{align}
\end{subequations}
where
\begin{equation*}
  \Lambda_j = \mathrm{diag}\left\{\lambda_{n_j +1}^{(j)},\ldots,\lambda_{N}^{(j)}\right\}\in \RR^{(N-n_{j})\times (N-n_{j})}
\end{equation*}
with $\lambda_{p}^{(j)},\, p = n_j +1, \ldots, N,$ being the nonzero eigenvalues of
the matrix $\HH_{\sigma(t_j)}$.

For notational simplicity, denote
\begin{align}\label{eq-M}
M_j = \mathrm{diag}\left\{A - \lambda_{n_j +1}^{(j)}BK, \ldots, A - \lambda_{N}^{(j)}BK\right\}.
\end{align}
Then, it is clear that the solution of system \eqref{eq-x1x2} satisfies
\begin{align}\label{eq-x-hat-solution}
	\begin{bmatrix} \hat{x}_1(t) \\ \hat{x}_2(t) \end{bmatrix} =
	\begin{bmatrix} I_{n_j} \otimes \mathbf{e}^{A(t-t_j)} & \zero \\ \zero & \mathbf{e}^{M_j(t-t_j)} \end{bmatrix}
	\begin{bmatrix} \hat{x}_1(t_j) \\ \hat{x}_2(t_j) \end{bmatrix} &\notag\\
	 t \in [t_j, t_{j+1}], \  j=0,1,2,\ldots. &
\end{align}
From \eqref{eq-x-hat} and \eqref{eq-x-hat-solution}, we have that the solution $\bar{x}(t)$
of system \eqref{eq-bar-x-rewrite} satisfies
\begin{align*}
	\bar{x}(t) = \left(P_{\sigma(t_j)} \otimes \mathbf{e}^{A(t-t_j)}+\bar{\Gamma}_{\sigma(t_j)} \mathbf{e}^{M_j(t-t_j)} \bar{\Gamma}_{\sigma(t_j)}^T \right) \bar{x}(t_j) &\notag\\
 t \in [t_j, t_{j+1}], \quad j=0,1,2,\ldots &
\end{align*}
where $P_{\sigma(t_j)}$ is defined in \eqref{eq-P-sigma-def}
and
\begin{align}\label{eq-bar-gamma}
	\bar{\Gamma}_{\sigma(t_j)} = \Gamma_{\sigma(t_j)} \otimes I_{n}.
\end{align}

Now, let
\begin{equation*}
  \tau_j = t_{j+1} - t_j, \quad  j =0,1,2,\ldots
\end{equation*}
which clearly satisfies  $\tau_j \ge \tau,\, \forall\, j =0,1,2,\ldots$.
Further, denote
\begin{align}\label{eq-Xi}
	\Xi_j(t) = \Psi_j(t) + \Phi_j(t) , \quad  t \in [0, \tau_j], \quad j=0,1,2,\ldots
\end{align}
where $\Psi_j(\cdot),\,\Phi_j(\cdot) \in \RR^{Nn \times Nn}$ are given as follows:
\begin{subequations}\label{eq-PsiPhi}
\begin{align}
	\Psi_j(t) &= P_{\sigma(t_j)} \otimes \mathbf{e}^{A  t} \label{eq-Psi}  \\
\Phi_j( t) &= \bar{\Gamma}_{\sigma(t_j)} \mathbf{e}^{M_j  t} \bar{\Gamma}_{\sigma(t_j)}^T \label{eq-Phi}.
\end{align}
\end{subequations}
Using the above notation, we have
\begin{equation*}
  \bar{x}(t_{j}) = \left( \prod_{r=0}^{j-1}\Xi_{r}(\tau_r)\right)\bar{x}(0), \quad j=0,1,2,\ldots.
\end{equation*}
Moreover, for any $t \ge 0$, there exists some nonnegative integer $j$
such that $t_j \le t < t_{j+1}$.
Thus, we have
\begin{align}\label{eq-bar-x-solution-Xi}
	\bar{x}(t) = \Xi_{j}(t - t_{j}) \left( \prod_{r=0}^{j-1}\Xi_{r}(\tau_r)\right)\bar{x}(0), \quad t\ge 0.
\end{align}

\subsection{Design of Feedback Gain Matrix $K$}

In this subsection, we present the design of the feedback gain matrix $K$
for the distributed control law \eqref{eq-control-law}.
In particular, the design of $K$ depends on a uniform lower bound for the nonzero
eigenvalues of the leader-follower matrix $\HH_{\sigma(t)}$ of the switching graph $\bar\GG_{\sigma(t)}$.

It is well known that the Laplacian of an undirected and connected graph
is symmetric, positive semi-definite, and has exactly one zero eigenvalue.
Moreover, an explicit lower bound for the nonzero eigenvalues of the Laplacian
can be obtained by the following lemma.

\begin{Lemma}\label{Lemma-Laplacian-Lower-Bound}
  Let $\LL_\mathrm{c} \in \RR^{N \times N}$
  be the Laplacian of an undirected and connected graph $\GG_\mathrm{c}$.
  Then the nonzero eigenvalues of $\LL_\mathrm{c}$
  are not less than
  \begin{equation*}
\lambda_{L} : =  \frac{4}{N(N-1)}>0.
\end{equation*}
\end{Lemma}

\begin{Remark}
Lemma \ref{Lemma-Laplacian-Lower-Bound} follows from \cite[Theorem 4.2]{Mohar-1991GC}
and is of critical use in \cite[Lemma 2]{WangZhuFeng2019TAC} to obtain a
uniform lower bound for the nonzero eigenvalues of
the Laplacian $\LL_{\sigma(t)}$ of the switching graph $\GG_{\sigma(t)}$.
In our case, we have to establish a uniform lower bound for the nonzero
eigenvalues of the leader-follower matrix $\HH_{\sigma(t)}$ of the switching graph $\bar\GG_{\sigma(t)}$.
To this end, we need to extend Lemma \ref{Lemma-Laplacian-Lower-Bound}
to the following form.
\end{Remark}

\begin{Lemma}\label{lemma-connected-H}
Consider the following matrix:
\begin{align*}
	\HH  =  \LL +\Delta  \in \RR^{N \times N}
\end{align*}
where $\LL \in \RR^{N \times N}$ is the Laplacian
of an undirected graph $\GG$
and $\Delta  \in \RR^{N \times N}$ is a $(0,1)$-diagonal matrix. 
Then the nonzero eigenvalues of $\HH$ are not less than
\begin{align}\label{eq-lower-bound-H}
 \lambda_{H}:= \frac{4}{N\left(N^2-N+4\right)} >0.
\end{align}
\end{Lemma}

\begin{Proof}
If $N = 1$, then the lower bound \eqref{eq-lower-bound-H} holds trivially,
since $\HH = \Delta$ does not have nonzero eigenvalues when $\Delta = 0$ 
and the eigenvalue of $\HH$ is equal to $\frac{4}{N\left(N^2-N+4\right)} = 1$ when $\Delta = 1$. 

For  $N >1$, we divide the proof into two steps.

{\bf Step 1:} In this step, we assume that the graph $\mathcal{G}$ is connected
as $\mathcal{G}_c$ and $\LL=\LL_\mathrm{c}$.
If  $\Delta = \mathbf{0}_{N \times N}$, 
then the result follows from Lemma \ref{Lemma-Laplacian-Lower-Bound}
by noting that
\begin{equation*}
\frac{4}{N(N-1)} >\frac{4}{N\left(N^2-N+4\right)}.
\end{equation*}

Next, assume $\Delta =\Delta_{\mathrm{c}}$ 
is a $(0,1)$-diagonal matrix with only one $1$ on its diagonal.
Note that any unit vector $v \in \RR^{N}$
can be written as
\begin{align}\label{eq-eta}
	v = c_1\frac{\mathbf{1}_{N}}{\sqrt{N}} + c_2\zeta
\end{align}
where $\zeta \in \RR^{N}$ is a unit vector that satisfies $\mathbf{1}_{N}^T \zeta = 0$, and $c_{1}, c_{2}$ are constants such that
$c_1^2 + c_2^2 = 1$.
Since the graph $\GG_\mathrm{c}$ is connected,
the Laplacian $\LL_\mathrm{c}$ has exactly one zero eigenvalue.
It follows from Lemma \ref{Lemma-Laplacian-Lower-Bound} that
\begin{align}\label{eq-lambda2}
	\zeta^T \LL_\mathrm{c}  \zeta \ge \lambda_{L}.
\end{align}

To complete the proof of this step, we only need
to evaluate a lower bound for the following quadratic form:
\begin{align}\label{eq-f}
	f := v^T \HH  v
\end{align}
where $\HH = \LL_\mathrm{c} + \Delta_\mathrm{c}$.
For this purpose, let
\begin{align}\label{eq-xi-part}
	\zeta = \mathrm{col}\left(\zeta_1, \ldots, \zeta_N\right), \quad \zeta_i \in \RR, \quad i=1,2,\ldots,N
\end{align}
and, for simplicity, assume
\begin{equation}\label{eq-Delta-o-def}
  \Delta_\mathrm{c} =\begin{bmatrix}
                    1 & \zero \\
                    \zero & \zero
                  \end{bmatrix} \in \RR^{N \times N}.
\end{equation}
From the definition of $\zeta$, we have
\begin{align*}
	\zeta_1 + \zeta_2 +\cdots + \zeta_N &= 0 \\
	\zeta_1^2 + \zeta_2^2 + \cdots + \zeta_N^2 &= 1.
\end{align*}
Then, it follows that
\begin{align*}
  \zeta_{1}^{2} & = (\zeta_{2}+\cdots+\zeta_{N})^{2}\\
   & \le (N-1)\left( \zeta_{2}^{2}+\cdots+\zeta_{N}^{2}\right)\\
   & = (N-1)(1-\zeta_{1}^{2}).
\end{align*}
Thus, we can obtain
\begin{align}\label{eq-xi-1-range}
	-\sqrt{\frac{N-1}{N}} \le \zeta_1 \le \sqrt{\frac{N-1}{N}}.
\end{align}

From  \eqref{eq-eta} to \eqref{eq-Delta-o-def},
we have
\begin{align*}
	f &= \left(c_1\frac{\mathbf{1}_{N}}{\sqrt{N}} + c_2\zeta \right)^T \left(\LL_\mathrm{c} +\Delta_\mathrm{c} \right) \left(c_1\frac{\mathbf{1}_{N}}{\sqrt{N}} + c_2\zeta \right) \notag\\
	&\ge  c_2^2 \zeta_1^2 + \frac{2 c_1 c_2 }{\sqrt{N}}\zeta_1 + \frac{c_{1}^{2}}{N} + c_2^2 \lambda_{L} =: g(\zeta_{1}).
\end{align*}
To obtain a lower bound for the above function $g(\zeta_{1})$ and hence,
a lower bound for $f$, consider three different cases.

\emph{Case 1:} If $c_{2}=0$, then
\begin{equation}\label{eq-f-bound-1}
  f \ge g(\zeta_{1})=\frac{c_{1}^{2}}{N}=\frac{1}{N}.
\end{equation}

If $c_{2} \ne 0$, the minimum of $g(\zeta_{1})$
would be attained when $\zeta_{1}=-\frac{c_{1}}{\sqrt{N}c_{2}}$.
However, in light of the range of $\zeta_{1}$ given in \eqref{eq-xi-1-range},
we need to further distinguish the following two cases.

\emph{Case 2:} If $\left|-\frac{c_{1}}{\sqrt{N}c_{2}} \right|\le \sqrt{\frac{N-1}{N}}$,
then
\begin{equation*}
     f  \ge g(\zeta_{1}) \ge g\left(-\frac{c_{1}}{\sqrt{N}c_{2}}\right)=c_{2}^{2}\lambda_{L}. \\
\end{equation*}
Since $\left|-\frac{c_{1}}{\sqrt{N}c_{2}} \right|\le \sqrt{\frac{N-1}{N}}$ also implies that
$c_{2}^{2}\ge \frac{1}{N}$, we have
\begin{equation}\label{eq-f-bound-2}
     f  \ge g(\zeta_{1}) \ge  \frac{\lambda_{L}}{N}. \\
\end{equation}

\emph{Case 3:} If $\left|-\frac{c_{1}}{\sqrt{N}c_{2}}\right| >\sqrt{\frac{N-1}{N}}$,
then
\begin{align}\label{eq-f-bound-3}
     f  \ge g(\zeta_{1}) & \ge g\left(\pm\sqrt{\frac{N-1}{N}}\right) \notag\\
     & = \frac{1}{N}c_{1}^{2} -\frac{2\sqrt{N-1}}{N}c_{1}c_{2}+ \left(\frac{N-1}{N}+\lambda_{L}\right)c_{2}^{2} \notag\\
     &=\begin{bmatrix}
         c_{1} & c_{2}
       \end{bmatrix}\begin{bmatrix}
                      \frac{1}{N} & \pm\frac{\sqrt{N-1}}{N} \notag\\
                       \pm\frac{\sqrt{N-1}}{N} &  \frac{N-1}{N}+\lambda_{L}
                    \end{bmatrix} \begin{bmatrix}
                                    c_{1} \\
                                    c_{2}
                                  \end{bmatrix} \notag\\
       & \ge \min\left\{ \frac{\lambda_{L} + 1 \pm \sqrt{(\lambda_{L}+1)^2 - \frac{4\lambda_{L}}{N}}}{2}\right\} \notag\\
		& = \frac{\lambda_{L} + 1 - \sqrt{(\lambda_{L}+1)^2 - \frac{4\lambda_{L}}{N}}}{2}
		\ge \frac{\lambda_{L}}{N(\lambda_{L}+1)}.
\end{align}

By combining \eqref{eq-f-bound-1}, \eqref{eq-f-bound-2}, and \eqref{eq-f-bound-3},
we obtain a lower bound for $f$ as follows:
\begin{align*}
	f &\ge \min\left\{\frac{1}{N}, \frac{\lambda_{L}}{N}, \frac{\lambda_{L}}{N(\lambda_{L}+1)}\right\} \\
	&= \frac{\lambda_{L}}{N(\lambda_{L}+1)} = \frac{4}{N(N^2-N+4)}
\end{align*}
which is the expression given by \eqref{eq-lower-bound-H}.

When $\Delta_\mathrm{c}$ is any other $(0,1)$ diagonal matrix with only one $1$ on its diagonal, the lower bound \eqref{eq-lower-bound-H}
can be similarly derived.

Now let $\Delta $ be any $(0,1)$ diagonal matrix with multiple $1$'s on its diagonal.
Since $\LL_\mathrm{c} +\Delta \ge \LL_\mathrm{c} +\Delta_\mathrm{c}$, 
the lower bound \eqref{eq-lower-bound-H} also holds.

{\bf Step 2:} In this step, we consider a disconnected graph $\GG$. 
By relabeling the nodes of $\GG$, we can partition its Laplacian   $\LL$
into the following block diagonal form:
\begin{align*}
\LL  = \mathrm{block\,diag}\left\{\LL_{1}, \ldots, \LL_{q}, 0, \ldots, 0 \right\}
\end{align*}
where each $0$ corresponds to an isolated node,
and each $\LL_i,\, i=1,\ldots,q$,
corresponds to a connected subgraph of the graph $\GG$ 
with $N_i$ nodes where $1 \leq N_i \leq N$.
By Step 1, for any  $(0,1)$-diagonal matrix $\Delta_i$ of dimension $N_i$, 
the nonzero eigenvalues of $\LL_i + \Delta_i$ are not less than 
$\lambda_{Hi} = \frac{4}{N_i \left(N^2_i -N_i+4\right)}$. 

Since the value of $\frac{4}{N\left(N^2-N+4\right)}$ decreases as $N$ increases
and $\frac{4}{N\left(N^2-N+4\right)}\le 1$,
the lower bound \eqref{eq-lower-bound-H} also holds for a disconnected graph $\GG$.

The overall proof is thus complete.
\end{Proof}

Let us consider the switching leader-follower matrix $\HH_{\sigma(t)}$ over $[0,\infty)$.  
Since, for any $t \geq 0$, 
\begin{equation*}
 \HH_{\sigma(t)} = \LL_{\sigma(t)}   +\Delta_{\sigma(t)} \in \RR^{N \times N}
\end{equation*}
and $\Delta_{\sigma(t)}\in \RR^{N \times N}$ is some $(0,1)$-diagonal matrix, Lemma \ref{lemma-connected-H} leads to a uniform lower bound for the
nonzero eigenvalues of $\HH_{\sigma(t)}$ for all $t\ge 0$, as stated below.

\begin{Lemma}\label{lemma-general-H}
Under Assumption \ref{Ass-undirected},
the nonzero eigenvalues of the leader-follower matrix $\HH_{\sigma(t)} \in \RR^{N \times N}$
of the switching graph $\bar\GG_{\sigma(t)}$
are not less than
\begin{align*}
 \lambda_{H}= \frac{4}{N\left(N^2-N+4\right)}
\end{align*}
for all $t\ge 0$.
\end{Lemma}

Now, with the establishment of Lemma \ref{lemma-general-H},
we are ready to present the design of the feedback gain matrix $K$
for the distributed control law \eqref{eq-control-law}.

Similar to the design in \cite{WangZhuFeng2019TAC},
given any $\alpha > 0$ and  $t^* > 0$, we define a weighted controllability Gramian as follows:
\begin{align}\label{eq-Gramian}
	W_\mathrm{c}(\alpha, t^*): = \int_0^{t^*} \mathbf{e}^{-\alpha t} \cdot \mathbf{e}^{-\frac{A}{2} t} B B^T \mathbf{e}^{-\frac{A^T}{2} t} dt \in \RR^{n \times n}.
\end{align}
Under Assumption \ref{Ass-stabilizable}, the weighted controllability Gramian $W_\mathrm{c}(\alpha, t^{*})$ is positive definite.
Then, we design the feedback gain matrix $K$ as
\begin{align}\label{eq-K}
	K = \mu B^T W_\mathrm{c}^{-1}(\alpha, t^*) \in \RR^{m \times n}
\end{align}
where  $\mu\ge\frac{1}{\lambda_{H}}$
and $\lambda_{H}>0$ is defined in \eqref{eq-lower-bound-H}.

\subsection{Proof of Theorem \ref{Theorem}}

As we have noted in Section \ref{Section-Problem},
Problem \ref{Problem} is solvable if it can be shown that the solution $\bar{x}(t)$
of system \eqref{eq-der-L0-otimes-x} satisfies
$\lim_{t \to\infty}\bar{x}(t)=0$ exponentially.
Moreover, we have derived in Subsection \ref{Subsection-Problem-Conversion}
 the closed-form expression of the solution $\bar{x}(t)$.
Before proceeding to the proof of Theorem \ref{Theorem},
we present two more technical lemmas regarding the closed-form expression of the solution $\bar{x}(t)$ in \eqref{eq-bar-x-solution-Xi}.

\begin{Lemma}\label{lemma-Phi}
Consider $\Phi_j(t)$ in \eqref{eq-Phi}.
Under Assumptions~\ref{Ass-stabilizable} and \ref{Ass-undirected}, 
given any $\alpha>0$ and $t^* > 0$,  
let the feedback gain matrix $K$ be given by \eqref{eq-K} and denote the maximum and the minimum eigenvalue of $W_\mathrm{c}(0, t^*)$ by $\lambda_M$ and $\lambda_m$, respectively.
Then, we have
$$\left\|\Phi_j( t)\right\| \leq C_0(t^*) \mathbf{e}^{-\alpha   t}, \ \forall\,   t \in [0, \tau_j],\ \forall\, j =0,1,2,\ldots $$
where $C_0(t^*) = \sqrt{\frac{\lambda_M}{\lambda_m}}\mathbf{e}^{\frac{\alpha}{2}t^*} \ge 1$.
\end{Lemma}

\begin{Proof}
Let us first rephrase  \cite[Lemma 3]{WangZhuFeng2019TAC} as follows: 
Under Assumption \ref{Ass-stabilizable}, 
for any $\lambda \geq \lambda_H$,  let
$K$ be given by \eqref{eq-K}. 
Then,
\begin{equation}\label{lemma3of2019TAC}
  \left\| \mathbf{e}^{\left(A - \lambda BK\right)t} \right\| \leq C_0(t^*) \mathbf{e}^{-\alpha t}, \quad \forall\, t \ge 0.
\end{equation}

By Lemma \ref{lemma-general-H}, 
under Assumption \ref{Ass-undirected}, 
we have $\lambda_p^{(j)} \ge \lambda_H$, 
for all $p = n_j+1, \ldots, N$, and all $j=0,1,2,\ldots$.
Thus, it follows from  \eqref{lemma3of2019TAC} that
\begin{equation*}
  \left\| \mathbf{e}^{\left(A - \lambda_p^{(j)} BK\right)t} \right\| \leq C_0(t^*) \mathbf{e}^{-\alpha t}, \quad \forall\, t \ge 0.
\end{equation*}

From the definitions of $M_j$ in \eqref{eq-M}, we have
\begin{align*}
	\left\|\mathbf{e}^{M_j t}\right\| \leq C_0(t^*) \mathbf{e}^{-\alpha t}, \quad \forall\, t \ge 0, \quad \forall\, j =0,1,2,\ldots.
\end{align*}
Consider $\Phi_j(t)$ in \eqref{eq-Phi}. We have
\begin{align*}
	\left\|\Phi_j(t)\right\| &\le \left\|\bar{\Gamma}_{\sigma(t_j)}\right\| \left\|\mathbf{e}^{M_j t}\right\| \left\|\bar{\Gamma}_{\sigma(t_j)}^T\right\| \\
	&\le C_0(t^*) \mathbf{e}^{-\alpha t}, \quad \forall\,  t \in [0, \tau_j], \ \forall\, j =0,1,2,\ldots.
\end{align*}
The proof is thus complete.
\end{Proof}

To continue, under the condition \eqref{eq-restriction-A} in Theorem \ref{Theorem},
we let $\lambda^*>0$ be any fixed number satisfying
\begin{equation}\label{eq-lambda-*-selection}
  \lambda_{\max}(A) < \lambda^* < - \frac{\ln \delta}{T_\mathrm{c}}.
\end{equation}
Then, it is clear that there exists some $C_1\ge 1$ such that
\begin{align}\label{eq-expA}
	\left\| \mathbf{e}^{A  t} \right\| \le C_1 \mathbf{e}^{\lambda^*  t}, \quad \forall\,  t \ge 0.
\end{align}

\begin{Lemma}\label{lemma-ratio}
Consider $\Xi_j(t)$ in \eqref{eq-Xi}.
Under Assumptions \ref{Ass-stabilizable} to \ref{Ass-connected},
and the condition \eqref{eq-restriction-A},
given any $t^* > 0$,
there exist a positive integer $\ell$ and a corresponding positive number $\alpha(\ell, t^*)$
such that,
if the feedback gain matrix $K$ is designed according to \eqref{eq-K}
with $\alpha > \alpha(\ell, t^*)$,
then
\begin{align}\label{eq-lemma-ratio}
	\left\|\prod_{r=j_k}^{j_{k+\ell}-1}\Xi_{r}(\tau_r)\right\| \le \rho, \qquad \forall\, k =0,1,2,\ldots
\end{align}
holds for some $\rho \in (0, 1)$.
\end{Lemma}

\begin{Proof}
See  Appendix \ref{Appendix-proof-lemma-ratio}.
\end{Proof}

\begin{Remark}
	Lemma \ref{lemma-ratio} in fact characterizes the relationship between the parameters $T_\text{c}$, $\delta$, and $\alpha$, and the constant $0<\rho<1$, which governs the exponential convergence rate.
More explicitly, from proof of Lemma \ref{lemma-ratio},
we can choose $\ell = \lfloor - \frac{\ln(C_1)}{\ln \delta + \lambda^* T_c} \rfloor+ 1$ so that \eqref{eq-l} is satisfied. Then, it follows from \eqref{eq-relation} that
$$\rho = C_1 \left(\delta \mathbf{e}^{\lambda^* T_\mathrm{c} }\right)^{\lfloor - \frac{\ln(C_1)}{\ln \delta + \lambda^* T_c} \rfloor+ 1} + C_3(t^*) \mathbf{e}^{-(\alpha + \lambda^*)\tau}.$$
	This expression shows that a smaller $T_\text{c}$ or a smaller $\delta$ leads to a faster exponential convergence rate, and increasing $\alpha$ can also accelerate the exponential convergence rate.
\end{Remark}

\medskip

Now, let us turn to the proof of Theorem \ref{Theorem}.

\noindent
\emph{Proof of Theorem 1:}
Following Lemma \ref{lemma-ratio},
given any $t^* > 0$,
there exist a positive integer $\ell$
and a corresponding positive number $\alpha(\ell,t^*)$
such that \eqref{eq-lemma-ratio} holds when the feedback gain matrix $K$ is designed as
\begin{align*}
  K = \mu B^T W_\mathrm{c}^{-1}\left(\alpha, t^*\right), \quad \mu\ge\frac{1}{\lambda_{H}}, \quad \alpha > \alpha(\ell,t^*).
\end{align*}

Fix a time instant $t\ge 0$.
Then, $t \in  [t_j, t_{j+1})$ for some nonnegative integer $j$.
In addition, $j_{\kappa\ell} \leq j < j_{(\kappa + 1)\ell}$ for some
nonnegative integer $\kappa$.
Thus, from the expression of $\bar{x}(t)$ in \eqref{eq-bar-x-solution-Xi}, we have
\begin{align}\label{eq-ratio-Xi}
	&\quad\, \left\|\bar{x}(t)\right\| \notag\\
	&\le \left\|\Xi_{j}(t - t_{j})\right\| \left\|\prod_{r=j_{\kappa\ell}}^{j-1}\Xi_{r}(\tau_r)\right\|\left\|\prod_{r=0}^{j_{\kappa\ell}-1}\Xi_{r}(\tau_r)\right\| \|\bar{x}(0)\|\notag\\
	&\le \left\|\Xi_{j}(t - t_{j})\right\| \left\|\prod_{r=j_{\kappa\ell}}^{j-1}\Xi_{r}(\tau_r)\right\|\prod_{k=0}^{\kappa-1}\left\|\prod_{r=j_{k\ell}}^{j_{(k+1)\ell}-1}\Xi_{r}(\tau_r)\right\| \|\bar{x}(0)\| \notag\\
	&\le  \left\| \Xi_{j}(t - t_{j})\right\| \prod_{r=j_{\kappa\ell}}^{j-1}\left\|\Xi_{r}(\tau_r)\right\| \rho^\kappa \|\bar{x}(0)\|
\end{align}
in which,
$\prod_{r=j_{\kappa\ell}}^{j-1}\left\|\Xi_{r}(\tau_r)\right\|$
is set to $1$ if $j=j_{\kappa\ell}$.
From \eqref{eq-P-sigma-P-norm}, \eqref{eq-Xi}, \eqref{eq-Psi}, and \eqref{eq-expA},
and by using Lemma \ref{lemma-Phi}, we have
\begin{align}\label{eq-sum-Xi}
	\left\|\Xi_j(t - t_{j})\right\| &\le \left\|\Phi_j(t - t_{j})\right\| + \left\|\Psi_j(t - t_{j})\right\| \notag\\
	&\le C_0(t^*) \mathbf{e}^{-\alpha (t - t_{j})} + C_1 \mathbf{e}^{\lambda^* (t - t_{j})} \notag\\
	&\le C_0(t^*) + C_1 \mathbf{e}^{\lambda^* T_\mathrm{c}}.
\end{align}
	Moreover, since $j - j_{\kappa\ell} < j_{(\kappa + 1)\ell} - j_{\kappa\ell} \leq \left\lceil\frac{T_\mathrm{c}}{\tau}\right\rceil\ell$,
\begin{align}\label{eq-prod-Xi}
	\prod_{r=j_{\kappa\ell}}^{j-1} \left\|\Xi_{r}(\tau_r)\right\| &\le \prod_{r=j_{\kappa\ell}}^{j-1} \left(C_0(t^*) \mathbf{e}^{-\alpha \tau_r} + C_1 \mathbf{e}^{\lambda^* \tau_r}\right) \notag\\
	&\le \prod_{r=j_{\kappa\ell}}^{j-1} \left(C_0(t^*) + C_1\right)\mathbf{e}^{\lambda^* \tau_r} \notag\\
	&\le \left(\left(C_0(t^*) + C_1\right)^{\left\lceil\frac{T_\mathrm{c}}{\tau}\right\rceil}\mathbf{e}^{\lambda^* T_\mathrm{c}}\right)^{\ell}.
\end{align}

By combining \eqref{eq-ratio-Xi}, \eqref{eq-sum-Xi}, and \eqref{eq-prod-Xi},
and letting
\begin{equation*}
  C_{2}(t^*) = \left(C_0(t^*) + C_1 \mathbf{e}^{\lambda^* T_\mathrm{c}}\right)\left(\left(C_0(t^*) + C_1\right)^{\left\lceil\frac{T_\mathrm{c}}{\tau}\right\rceil}\mathbf{e}^{\lambda^* T_\mathrm{c}}\right)^{\ell}
\end{equation*}
we have
\begin{align*}
	\left\|\bar{x}(t)\right\| &\le C_{2}(t^*)  \rho^\kappa  \left\|\bar{x}(0)\right\|
	= \frac{C_{2}(t^*)}{\rho} \rho^{\kappa+1} \left\|\bar{x}(0)\right\|\notag\\
	&\le \frac{C_{2}(t^*)}{\rho} \rho^{\frac{t}{T_\mathrm{c} \ell}} \left\|\bar{x}(0)\right\|
	= \frac{C_{2}(t^*)}{\rho} \mathbf{e}^{- \varrho t}\left\|\bar{x}(0)\right\|
\end{align*}
where $\varrho= \frac{\ln(\rho^{-1})}{T_\mathrm{c} \ell} > 0$.
Thus, $\lim_{t\to\infty}\bar{x}(t)=0$ exponentially
and the proof is complete.
\hfill $\Box$\par

\section{Dual Problem}\label{Section-observer}

In this section, we further study the dual problem of Problem \ref{Problem},
namely, the output-based distributed observer design problem for a linear leader system over jointly connected switching networks.

\subsection{Output-based Distributed Observer}

Consider a leader-follower multi-agent system consisting of one leader system
and $N$ follower subsystems, in which, the linear leader system with an output
is described as follows:
\begin{align}\label{eq-system-output}
	\dot x_{0} (t) = A x_{0} (t), \quad y_0 (t) = C x_{0} (t), \quad t \geq 0
\end{align}
where $x_{0} (t) \in \RR^n$ and $y_0 (t) \in \RR^m$ are the state and the output of the leader system, respectively,
$A \in \RR^{n \times n}$ is the system matrix which can be unstable,
and $C \in \RR^{m \times n}$ is the output matrix.

As before, given the switching graph $\bar \GG_{\sigma (t)} = (\bar{\mathcal{V}}, \bar{\mathcal{E}}_{\sigma(t)})$,
we associate node $0$ with the leader system \eqref{eq-system-output}
and each of the other $N$ nodes with each of the $N$ follower subsystems.
Suppose, at any time instant $t$,
only those follower subsystems who are the children of the leader system
can measure the output $y_{0}(t)$ of the leader system \eqref{eq-system-output}.
Then, the output-based distributed observer design problem
aims to design $N$ local observers,
one for each of the $N$ follower subsystems,
so that they can cooperatively estimate the full state $x_{0}(t)$
of the leader system \eqref{eq-system-output} while satisfying the communication constraints
imposed by the switching graph $\bar \GG_{\sigma (t)}$.

We propose a distributed observer candidate as follows:
\begin{equation}\label{eq-compensator}
	\dot \eta_i (t) = A \eta_i (t) + L C\sum_{j=0}^N a_{ij}(t) \left( \eta_j (t) - \eta_i(t) \right), \   i = 1, \ldots, N
\end{equation}
where $\eta_{0}(t)=x_{0}(t)$,
$\eta_i (t) \in \RR^n$
is the state of the $i$th local observer,
 $L \in \RR^{n \times m}$ is an observer gain matrix to be designed,
and $a_{ij} (t)$ are entries of the adjacency matrix $\bar{\mathcal{A}}_{\sigma(t)}$
of the switching graph $\bar \GG_{\sigma (t)}$.
Specifically, since $C\eta_{0}(t)=Cx_{0}(t)=y_{0}(t)$,
the $i$th local observer can make use of the output $y_{0}(t)$ of the leader system
only when $a_{i0}(t)=1$.
Therefore, the distributed observer candidate \eqref{eq-compensator} will be called an output-based distributed observer
for the leader system \eqref{eq-system-output}, as it only relies on the leader's output.

Now we describe the output-based distributed observer design problem as follows.

\begin{Problem}\label{Problem-observer}
Given the leader system \eqref{eq-system-output} and the switching graph $\bar \GG_{\sigma (t)}$,
design a distributed dynamic compensator of the form \eqref{eq-compensator} such that,
for any initial conditions $x_{0} (0) \in \RR^n$ and $\eta_i(0) \in \RR^n,\, i= 1, \ldots, N$,
the solutions of systems \eqref{eq-system-output} and \eqref{eq-compensator} exist over $[0, \infty)$
and satisfy
$\lim_{t \to \infty} (\eta_i (t) - x_{0} (t)) = 0,\, i = 1, \ldots, N$, exponentially.
\end{Problem}

In addition to Assumptions \ref{Ass-undirected} and \ref{Ass-connected}
on the switching graph $\bar \GG_{\sigma (t)}$,
another standard assumption on the leader system \eqref{eq-system-output}
for solving Problem \ref{Problem-observer} is given below.

\begin{Assumption}\label{Ass-observable}
	The matrix pair $\left(C, A\right)$ is observable.
\end{Assumption}

\subsection{Dual System}

Denote the estimation error between the $i$th local observer in \eqref{eq-compensator}
and the leader system \eqref{eq-system-output} as
\begin{equation*}\label{eq-bar-i-eta-def}
  \bar{\eta}_i (t) = \eta_i (t) - x_{0} (t), \quad  i = 1, \ldots, N
\end{equation*}
and let
\begin{equation*}\label{eq-bar-eta-def}
  \bar \eta(t) = \mathrm{col}\left(\bar{\eta}_1 (t), \ldots, \bar{\eta}_N (t)\right) \in \RR^{Nn}.
\end{equation*}
Similar to Problem \ref{Problem},
it can be obtained that
\begin{align}\label{eq-bar-eta}
	\dot{\bar\eta} (t) = \left(I_{N}\otimes A - \mathcal{H}_{\sigma(t)}\otimes LC \right)\bar\eta(t), \quad t\ge 0
\end{align}
and Problem \ref{Problem-observer} is solvable if and only if
it can be made that
$\lim_{t\to\infty}\bar\eta (t)=0$ exponentially.

We call the linear switched system \eqref{eq-bar-eta}
the dual system of  system \eqref{eq-der-L0-otimes-x}.
By Theorem \ref{Theorem}, under Assumptions \ref{Ass-stabilizable} to \ref{Ass-connected},
system \eqref{eq-der-L0-otimes-x} can be exponentially stabilized via designing the feedback gain matrix $K$.
Following the same approach, we will show that the dual system \eqref{eq-bar-eta}
can also be exponentially stabilized via designing the observer gain matrix $L$.

In particular, similar to what has been performed in Subsection \ref{Subsection-Problem-Conversion},
under Assumption \ref{Ass-undirected},
it can be derived that the solution $\bar\eta(t)$
of system \eqref{eq-bar-eta} satisfies
\begin{align*}
	\bar\eta(t) = \left(P_{\sigma(t_j)} \otimes \mathbf{e}^{A(t-t_j)}+\bar{\Gamma}_{\sigma(t_j)} \mathbf{e}^{M_j'(t-t_j)} \bar{\Gamma}_{\sigma(t_j)}^T \right) \bar\eta(t_j) &\notag\\
 t \in [t_j, t_{j+1}], \ j=0,1,2,\ldots &
\end{align*}
where $P_{\sigma(t_j)}$ and  $\bar{\Gamma}_{\sigma(t_j)}$ are defined in \eqref{eq-P-sigma-def} and \eqref{eq-bar-gamma}, respectively, and
\begin{align*}
M_j' = \mathrm{diag}\left\{A - \lambda_{n_j +1}^{(j)}LC, \ldots, A - \lambda_{N}^{(j)}LC\right\}.
\end{align*}

Denote
\begin{align}\label{eq-Xi'}
	\Xi_j'( t) = \Psi_j(  t) + \Phi_j'(  t) , \quad   t \in [0, \tau_j], \quad j=0,1,2,\ldots
\end{align}
where $\Psi_j(\cdot),\,\Phi'_j(\cdot) \in \RR^{Nn \times Nn}$ are given by
\begin{subequations}\label{eq-PsiPhi'}
\begin{align}
	\Psi_j (  t) &= P_{\sigma(t_j)} \otimes \mathbf{e}^{A  t} \label{eq-Psi'}  \\
\Phi_j'(  t) &= \bar{\Gamma}_{\sigma(t_j)} \mathbf{e}^{M_j'  t} \bar{\Gamma}_{\sigma(t_j)}^T \label{eq-Phi'}.
\end{align}
\end{subequations}
Then, for any $t \ge 0$, we have
$t_j \le t < t_{j+1}$ for some nonnegative integer $j$,
and hence
\begin{align}\label{eq-bar-eta-expression}
	\bar\eta(t) = \Xi_{j}'(t - t_{j}) \left( \prod_{r=0}^{j-1}\Xi_{r}'(\tau_r)\right)\bar\eta(0), \quad t\ge 0.
\end{align}

\begin{Remark}
  It is noted that the only difference between the closed-form expression of the solution $\bar{\eta}(t)$ in \eqref{eq-bar-eta-expression}
  for the dual system \eqref{eq-bar-eta} and
  the closed-form expression of the solution $\bar{x}(t)$ in \eqref{eq-bar-x-solution-Xi}
  for system \eqref{eq-der-L0-otimes-x}
  lies in that  $\bar{\eta}(t)$ depends on $M_{j}'=\mathrm{diag}\left\{A - \lambda_{n_j +1}^{(j)}LC, \ldots, A - \lambda_{N}^{(j)}LC\right\}$
  while $\bar{x}(t)$ depends on $M_j = \mathrm{diag}\left\{A - \lambda_{n_j +1}^{(j)}BK, \ldots, A - \lambda_{N}^{(j)}BK\right\}$.
  More explicitly, in $M_j$, the matrix pair $(A,B)$ is controllable and $K$ is to be designed,
  whereas, in $M_{j}'$, the matrix pair $(C,A)$ is observable and $L$ is to be designed.
\end{Remark}

\subsection{Solvability of Problem \ref{Problem-observer}}
Dual to \eqref{eq-Gramian}, given any $\alpha > 0$ and  $t^* > 0$,
we define a weighted observability Gramian as follows:
\begin{align*}
	W_\mathrm{o}(\alpha, t^*): = \int_0^{t^*} \mathbf{e}^{-\alpha t} \cdot \mathbf{e}^{-\frac{A^T}{2} t} C^T C \mathbf{e}^{-\frac{A}{2} t} dt \in \RR^{n \times n}.
\end{align*}
Under Assumption \ref{Ass-observable}, $W_\mathrm{o}(\alpha, t^{*})$ is positive definite.
Then, we design the observer gain $L$ matrix as
\begin{align}\label{eq-L}
	L = \mu W_\mathrm{o}^{-1}(\alpha, t^*) C^T\in \RR^{n \times m}
\end{align}
where $\mu \ge \frac{1}{\lambda_{H}}$
and $\lambda_{H}>0$ is defined in \eqref{eq-lower-bound-H}.

\begin{Lemma}\label{lemma-Phi'}
Consider $\Phi_j'( t)$ in \eqref{eq-Phi'}.
Under Assumptions~\ref{Ass-undirected} and \ref{Ass-observable},
given any $\alpha>0$ and $t^* > 0$,
let  the observer gain matrix $L$  be given by \eqref{eq-L} and denote the maximum and the minimum eigenvalue of $W_\mathrm{o}(0, t^*)$ by $\lambda_M '$ and $\lambda_m '$, respectively.
Then, we have
$$\left\|\Phi_j'(  t)\right\| \leq C_0'(t^*) \mathbf{e}^{-\alpha  t}, \ \forall\,   t \in [0, \tau_j],\ \forall\, j =0,1,2,\ldots $$
where $C_0'(t^*) = \sqrt{\frac{\lambda_M '}{\lambda_m '}}\mathbf{e}^{\frac{\alpha}{2}t^*} \ge 1$.
\end{Lemma}

\begin{Remark}
The proof of Lemma \ref{lemma-Phi'} follows similarly to that of Lemma \ref{lemma-Phi},
leveraging the duality between the observability of $(C,A)$
and the controllability of   $(A^{T},C^{T})$, along with the fact
 that $\left\|\mathbf{e}^{\left(A - \lambda_p^{(j)} LC\right) t}\right\|=\left\|\mathbf{e}^{\left(A^{T} - \lambda_p^{(j)} C^{T}L^{T}\right)  t}\right\|,
\,p = n_j+1, \ldots, N,\, j=0,1,2,\ldots $.
\end{Remark}

\begin{Lemma}\label{lemma-ratio'}
Consider $\Xi_j'( t)$ in \eqref{eq-Xi'}.
Under Assumptions \ref{Ass-undirected} to \ref{Ass-observable},
and the condition \eqref{eq-restriction-A},
given any $t^* > 0$,
there exist a positive integer $\ell'$ and a corresponding positive number $\alpha'(\ell', t^*) $
such that,
if the observer gain matrix $L$ is designed according to \eqref{eq-L}
with $\alpha > \alpha'(\ell', t^*)$,
then
\begin{align*}
	\left\|\prod_{r=j_k}^{j_{k+\ell}-1}\Xi_{r}'(\tau_r)\right\| \le \rho', \qquad \forall\, k =0,1,2,\ldots
\end{align*}
holds for some $\rho' \in (0, 1)$.
\end{Lemma}

\begin{Remark}
The proof of Lemma \ref{lemma-ratio'} is the same as the proof of Lemma \ref{lemma-ratio}.
Furthermore, by invoking Lemmas \ref{lemma-Phi'} and \ref{lemma-ratio'}
in the same way as that in the proof of Theorem \ref{Theorem},
we can conclude the solvability of Problem \ref{Problem-observer} as in the following Theorem \ref{Theorem-dual},
which can be viewed as the dual result of Theorem \ref{Theorem}.
\end{Remark}

\begin{Theorem}\label{Theorem-dual}
Under Assumptions \ref{Ass-undirected} to \ref{Ass-observable},
  Problem \ref{Problem-observer} is solvable if the condition \eqref{eq-restriction-A} holds.
\end{Theorem}

\begin{Remark}
  Similar to what has been demonstrated in Remark \ref{Remark-controllable},
  if Problem \ref{Problem-observer} is solvable, then Assumption \ref{Ass-observable}
  can be relaxed to that the matrix pair $(C,A)$ is detectable.
\end{Remark}

\begin{figure}
	\centering
	\subfigure[$\bar\GG_1$]{
		\begin{minipage}[t]{0.22\textwidth}
		\centering
		\includegraphics[width=0.9\linewidth, trim=0 0 0 0]{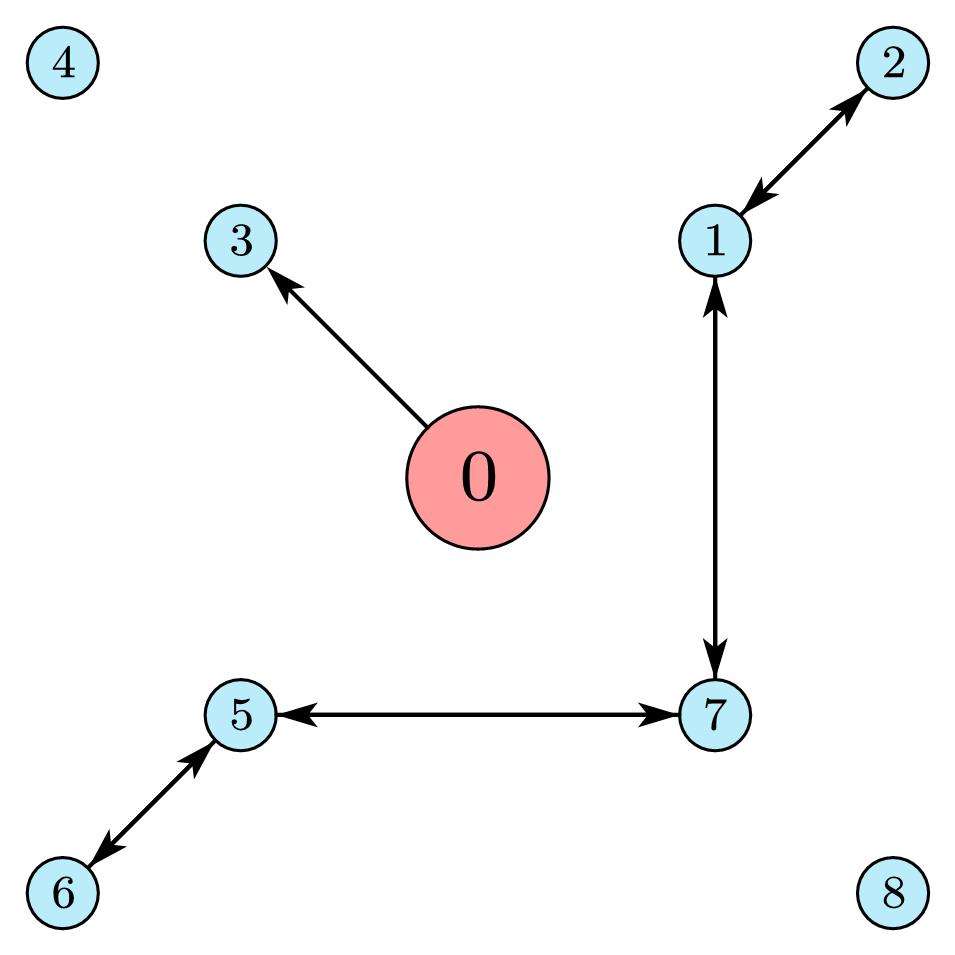}
		\end{minipage}
	}
	\subfigure[$\bar\GG_2$]{
		\begin{minipage}[t]{0.22\textwidth}
		\centering
		\includegraphics[width=0.9\linewidth, trim=0 0 0 0]{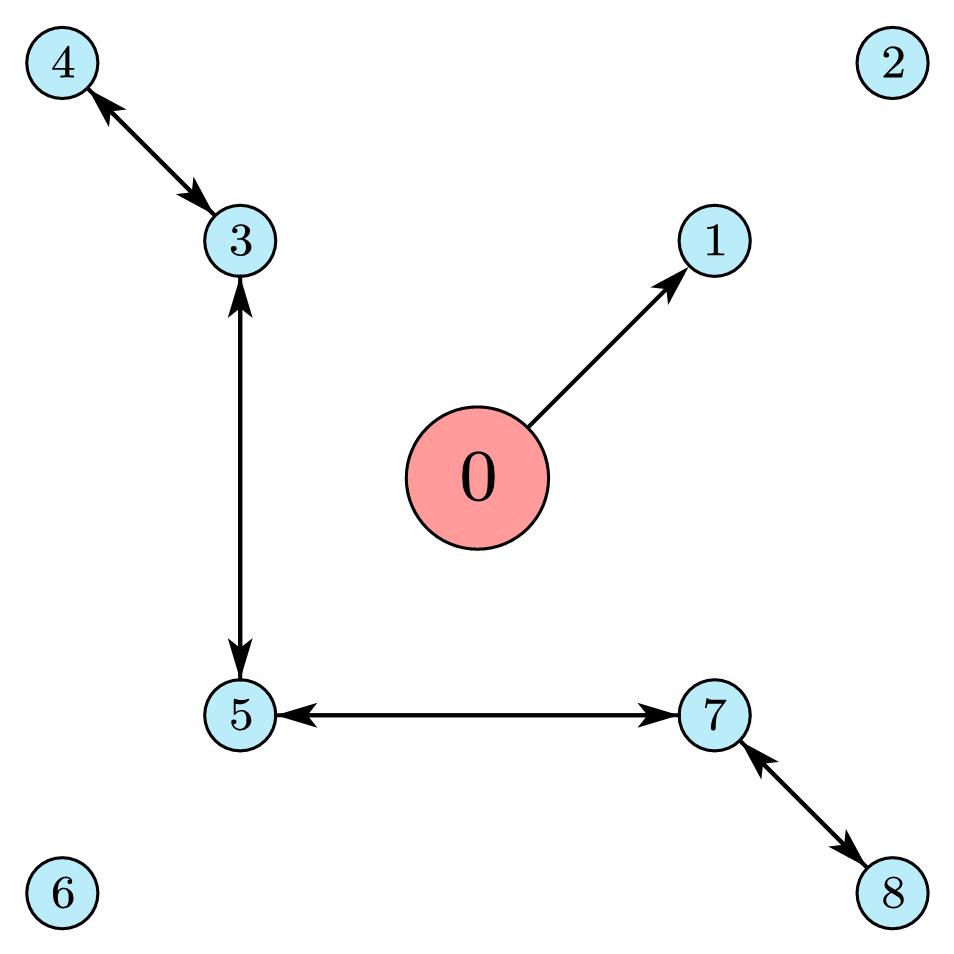}
		\end{minipage}
	}
	\caption{Switching graph $\bar \GG_{\sigma(t)}$.}\label{fig-G}
\end{figure}

\section{Numerical Examples}\label{Section-Example}

Consider a leader-follower multi-agent system consisting of one leader system
and eight follower subsystems  as follows:
\begin{subequations}\label{eq-example}
\begin{align}
	\dot x_0 (t) &= A x_0 (t), \qquad\  y_{0}(t)=Cx_{0}(t),&  t \geq 0 & \label{eq-example-leader}\\
	\dot x_i (t) &= A x_i (t) + B u_i (t), &  i = 1, \ldots, 8 &
\end{align}
\end{subequations}
where $x_0(t) \in \RR^3$ and  $y_{0}(t)\in \RR$ denote the state and the output of the leader system, respectively,
$x_i(t) \in \RR^3$ and $u_i(t) \in \RR$ denote the state and the control input of the $i$th follower subsystem, respectively, and
\begin{equation*}
   A = \begin{bmatrix} 2/3 & -25/6 & 0 \\ 4/3 & 0 & 1/3 \\ 0 & -1 & 1/6 \end{bmatrix}, \quad B = \begin{bmatrix} 0 \\ 0 \\ 1 \end{bmatrix}, \quad C=B^{T}.
\end{equation*}
In particular, the system matrix $A$ has eigenvalues at $\left\{0.195, 0.319 \pm \mathrm{j}2.40 \right\}$ and  is clearly unstable.
Moreover, it can  be verified that the matrix pair $(A,B)$ is controllable
and the matrix pair $(C,A)$ is observable.
Hence, Assumptions \ref{Ass-stabilizable} and \ref{Ass-observable} are satisfied.

\begin{figure}
	\centering
	\includegraphics[width=1\linewidth]{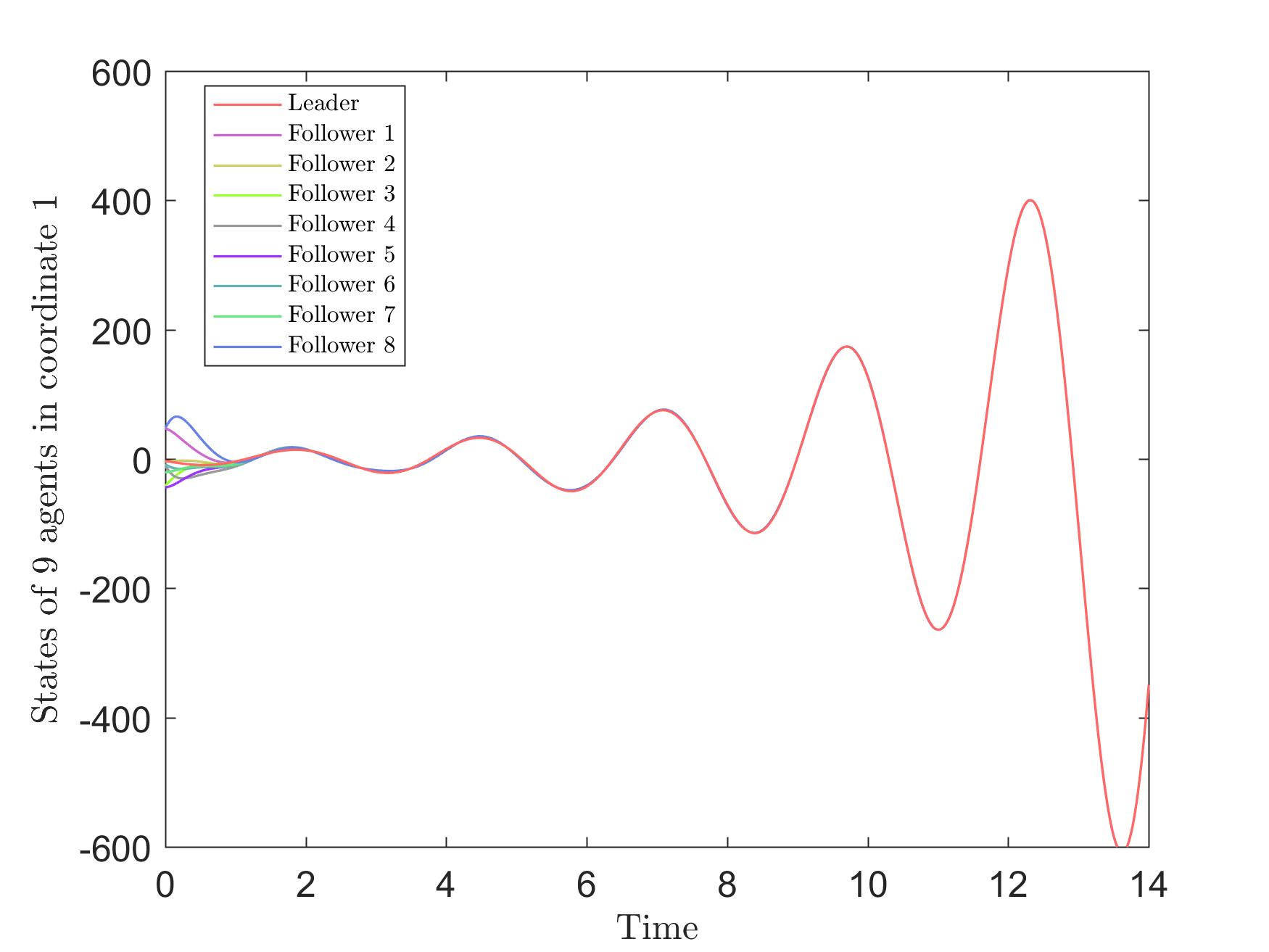}
	\caption{Profile on the first component of the states of the leader and followers.}\label{fig-results-1}
\end{figure}

\begin{figure}
	\centering
	\includegraphics[width=1\linewidth]{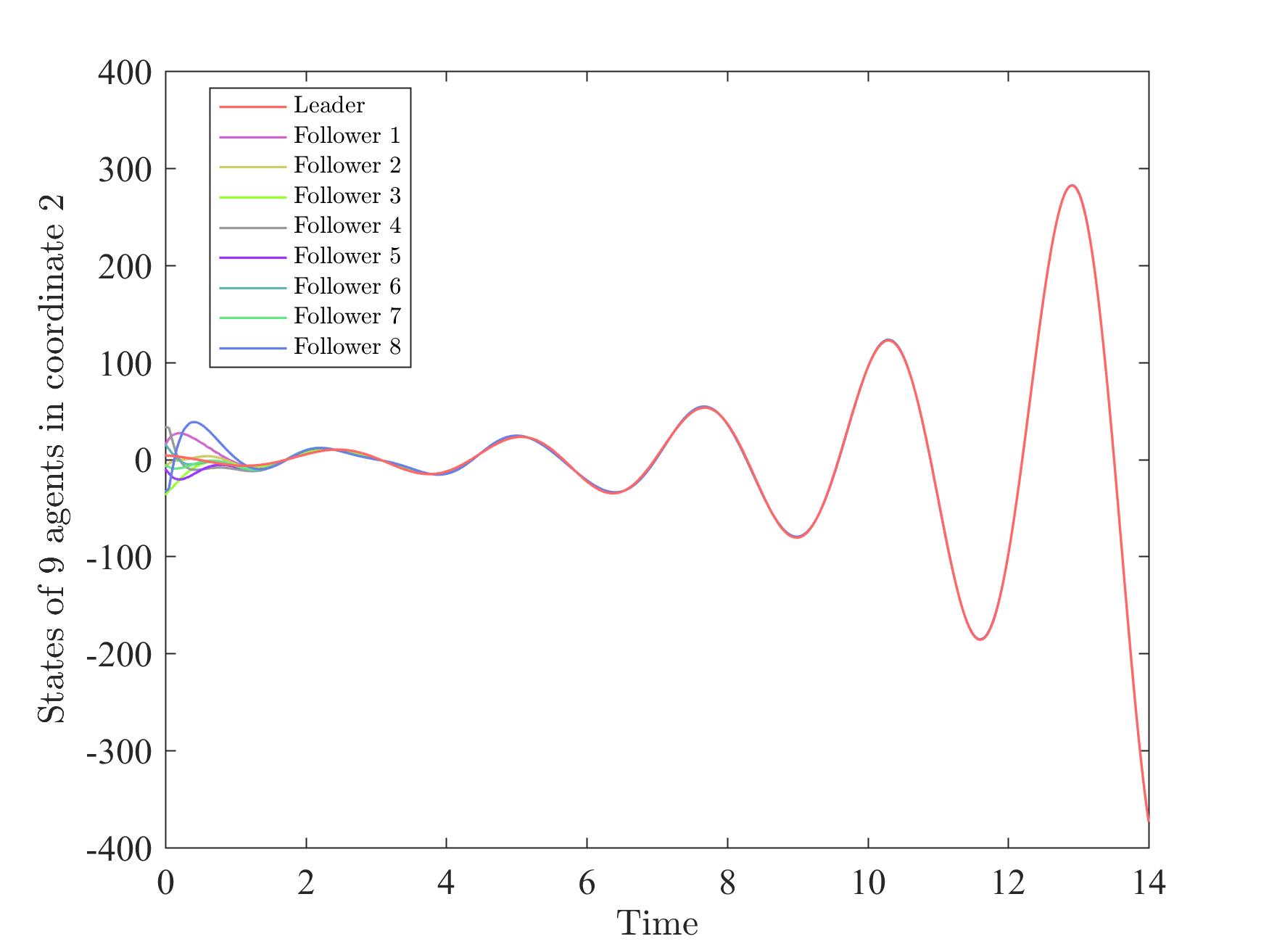}
	\caption{Profile on the second component of the states of the leader and followers.}\label{fig-results-2}
\end{figure}

\begin{figure}
	\centering
	\includegraphics[width=1\linewidth]{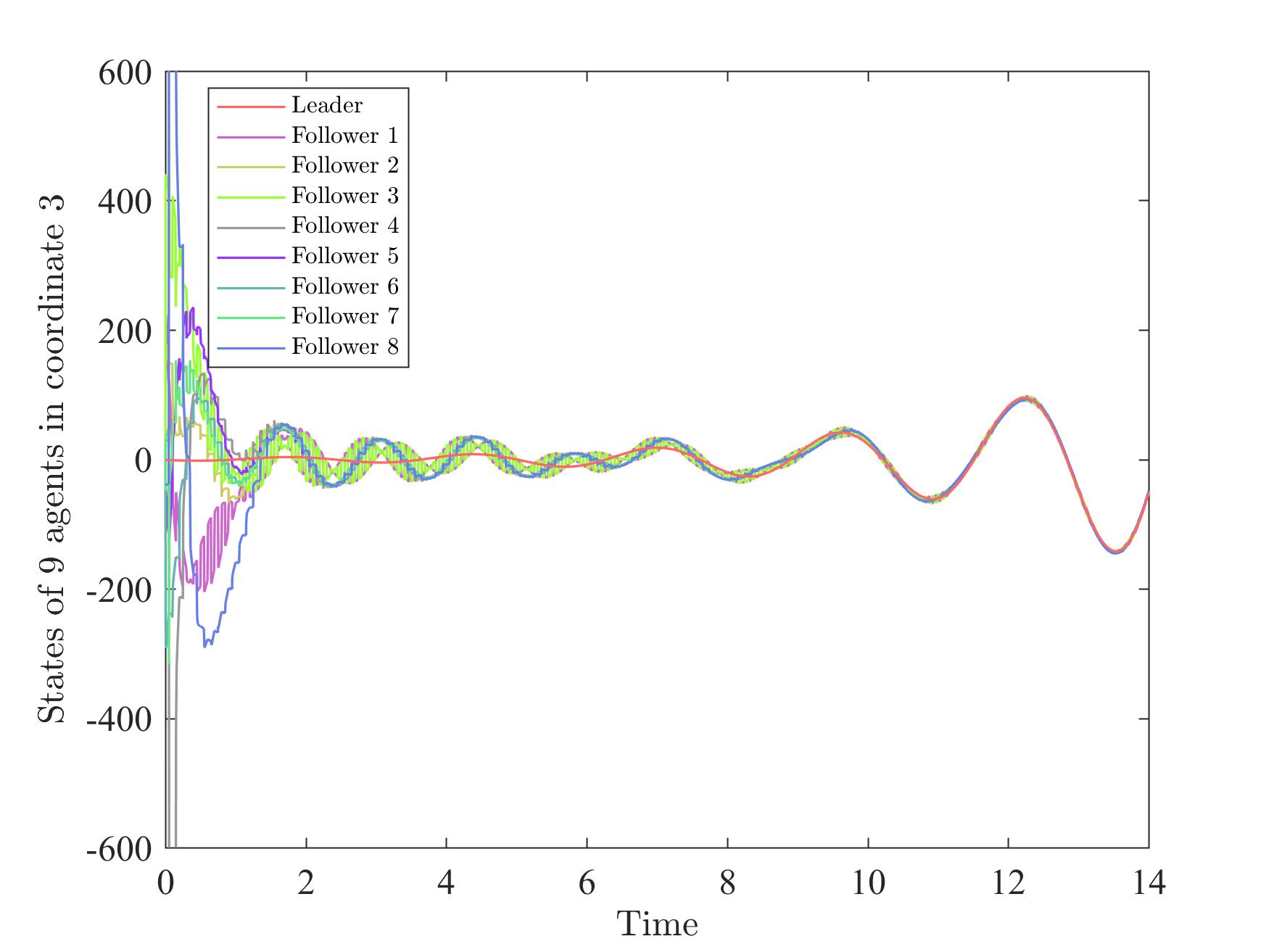}
	\caption{Profile on the third component of the states of the leader and followers.}\label{fig-results-3}
\end{figure}

\begin{figure}
	\centering
	\includegraphics[width=1\linewidth]{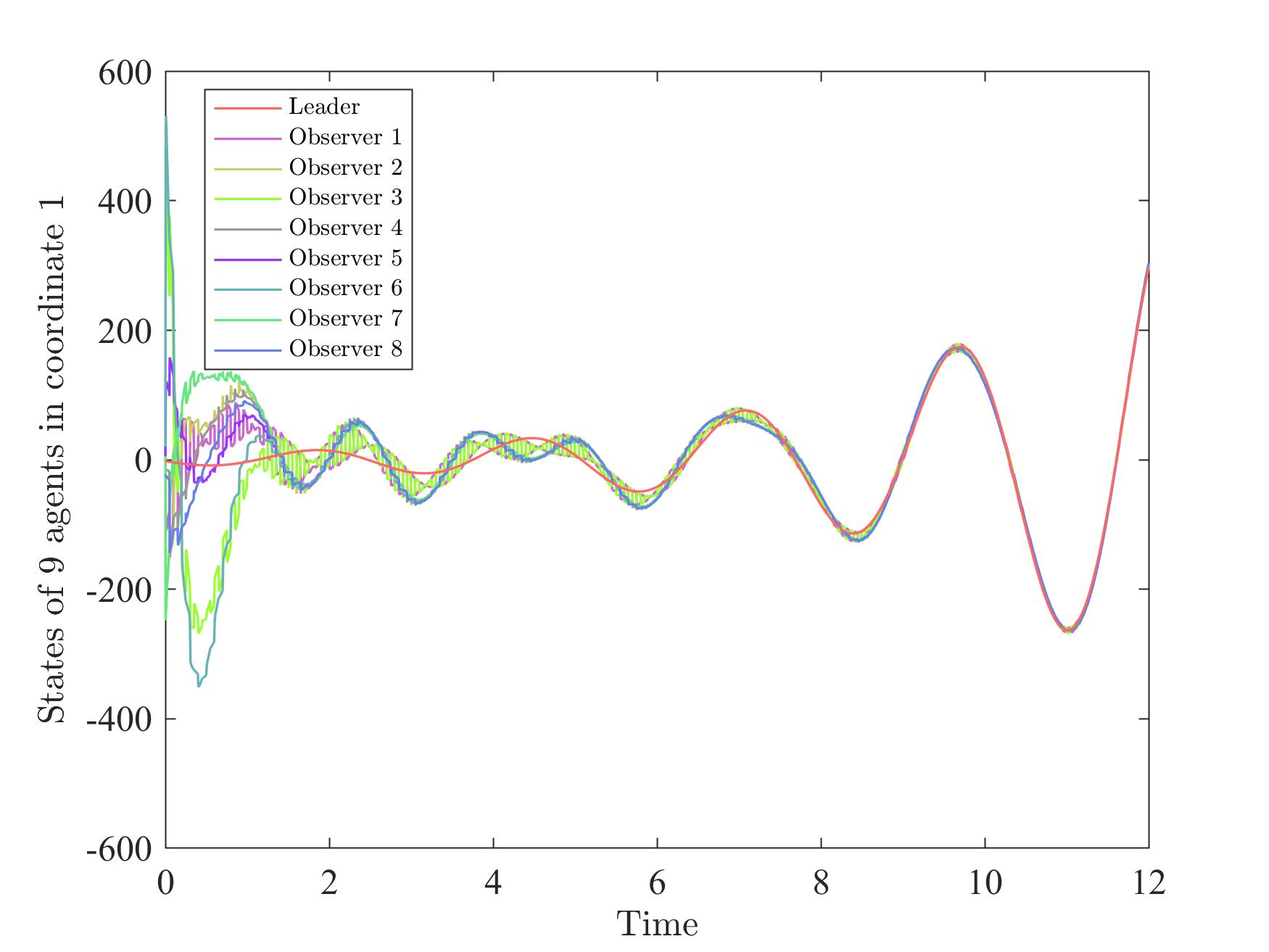}
	\caption{Profile on the first component of the states of the leader and observers.}\label{fig-results-observer-1}
\end{figure}

\begin{figure}
	\centering
	\includegraphics[width=1\linewidth]{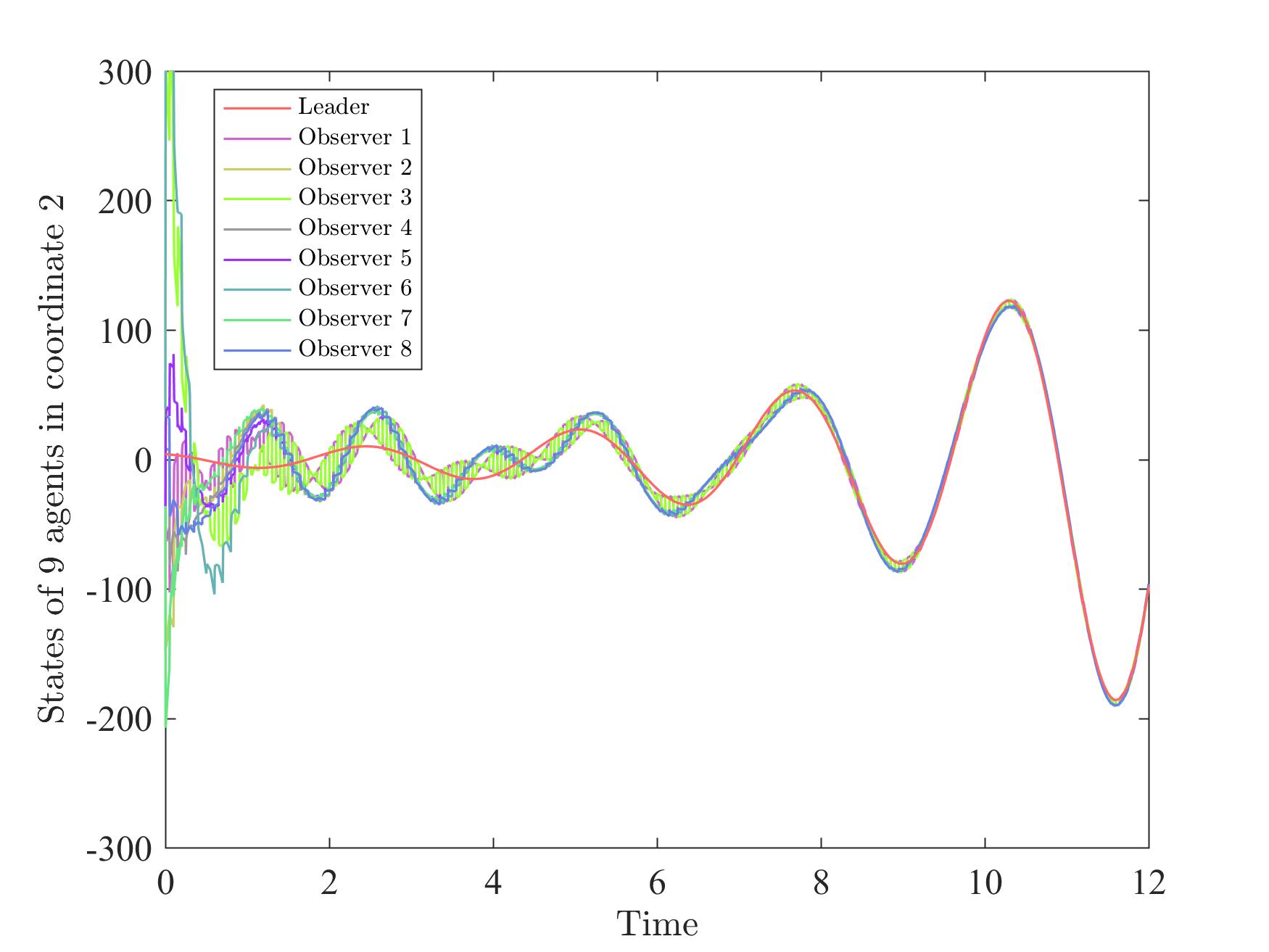}
	\caption{Profile on the second component of the states of the leader and observers.}\label{fig-results-observer-2}
\end{figure}

\begin{figure}
	\centering
	\includegraphics[width=1\linewidth]{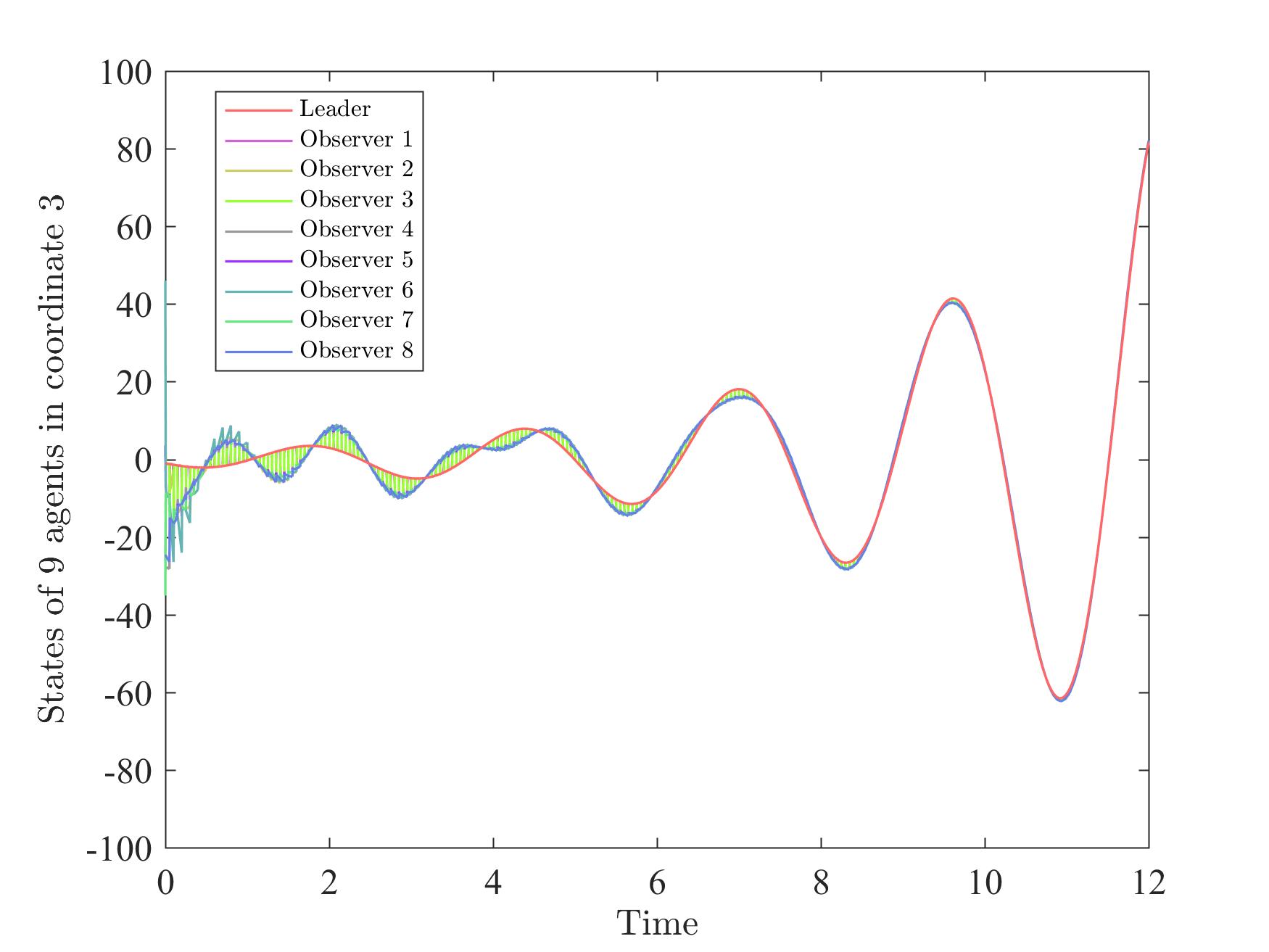}
	\caption{Profile on the third component of the states of the leader and observers.}\label{fig-results-observer-3}
\end{figure}

Let the switching communication network for the above leader-follower multi-agent system  be described by the switching graph $\bar \GG_{\sigma(t)}$ in Figure \ref{fig-G},  which is dictated by the following switching signal:
\begin{align*}
	\sigma(t) =
	\begin{cases}
		1,\enspace \mathrm{if} \enspace s T_\mathrm{c}  \leq t < \left(s + \frac{1}{2}\right) T_\mathrm{c} \\
		2,\enspace \mathrm{if} \enspace \left(s + \frac{1}{2}\right) T_\mathrm{c}  \leq t < \left(s + 1\right) T_\mathrm{c} \\
	\end{cases}
\end{align*}
where $s = 0, 1, 2, \ldots$, and $T_\mathrm{c} = 0.1$.
Then, it can be seen that Assumptions \ref{Ass-undirected} and \ref{Ass-connected} are satisfied.
In addition, the number $\delta$ as expressed in \eqref{eq-delta-def}
can be found as $0.863$.
Thus, it can be verified that $\lambda_{\max}(A) = 0.319 < - \frac{\ln \delta}{T_\mathrm{c}} = 1.47$, i.e., the condition \eqref{eq-restriction-A} holds.

On the one hand, by Theorem \ref{Theorem}, the leader-following exponential consensus problem
for the leader-follower multi-agent system \eqref{eq-example} over the above jointly connected switching network
can be solved by a distributed static state feedback control law of the form \eqref{eq-control-law}.
We design the feedback gain matrix $K$ as in \eqref{eq-K} with 
$\mu = \frac{1}{\lambda_H}$, $\alpha = 3$, and $t^* = 5$.
Simulation is performed with random initial conditions.
Figures \ref{fig-results-1} to \ref{fig-results-3}
show the trajectories of each component of the state of the leader system
and the states of the follower subsystems.
It can be observed that the states of the follower subsystems
all converge to the state of the leader system,
even though the trajectories of the leader system diverge exponentially.

On the other hand, by Theorem \ref{Theorem-dual},
an output-based distributed observer of the form \eqref{eq-compensator} can be designed
to exponentially estimate the state $x_{0}(t)$ of the leader system \eqref{eq-example-leader}
over the above jointly connected switching network.
We design the observer gain matrix $L$ as in \eqref{eq-L} with
$\mu = \frac{1}{\lambda_H}$, $\alpha = 3$, and $t^* = 5$.
Simulation is performed with random initial conditions.
Figures \ref{fig-results-observer-1} to \ref{fig-results-observer-3}
show the trajectories of each component of the state of the leader system and
the states of the local observers.
It can be seen that the states of the local observers
all converge to the state of the leader system,
even though the trajectories of the leader system diverge exponentially.


\section{Conclusions}\label{Section-Conclusions}

In this paper, we have resolved the long standing issue
that leader-following consensus over jointly connected switching networks
could only be achieved for marginally stable linear systems.
We have alleviated this stringent condition by allowing the system matrix to exhibit a certain degree of exponential instability.
This degree of instability is characterized by two quantities
derived from the jointly connected condition on the switching graph,
one of which is obtained by studying the product of the orthogonal projection matrices
onto the kernel of the leader-follower matrix of the switching graph.
Then, within this degree of exponential instability,
we have demonstrated the solvability of the leader-following exponential consensus problem for general linear multi-agent systems over jointly connected switching networks.
By exploiting duality,
we have further presented the design of an exponentially convergent output-based distributed observer for general linear leader systems over jointly connected switching networks, which is the dual problem to the leader-following exponential consensus problem.
Finally, we note that it would also be interesting to establish the discrete-time counterpart of the breakthrough made in this paper.


\appendices

\section{Notation on Graph}\label{Appendix-graph}
A graph $\GG := (\mathcal{V}, \mathcal{E})$ consists of a finite node set
$\mathcal{V} = \left\{1,2,\ldots,N\right\}$ and an edge set $\mathcal{E} \subseteq \mathcal{V} \times \mathcal{V}$.
An edge of $\mathcal{E}$ from node $j$ to node $i,\, j \ne i$, is denoted by $(j, i)$,
and node $i$ is called a child of node $j$.
The edge $(i, j)$ is called undirected if $(i, j) \in \mathcal{E}$ implies $(j, i) \in \mathcal{E}$.
The graph $\GG$ is called undirected if every edge in $\mathcal{E}$ is undirected.
If the graph contains a set of edges of the form
$\left\{(i_1, i_2), (i_2, i_3), \ldots, (i_{k-1}, i_k)\right\}$,
then this set is called a path from node $i_1$ to node $i_k$,
and node $i_k$ is said to be reachable from node $i_1$,
A graph is called strongly connected if there exists a path between any pair of nodes.
An undirected and strongly connected graph is called connected.

The adjacency matrix of a graph $\GG$ is a nonnegative matrix
$\mathcal{A} := [a_{ij}]^N_{i,j=1} \in \RR^{N \times N}$,
where $a_{ij} = 1$ if $(j, i) \in \mathcal{E}$ and $a_{ij} = 0$ otherwise.
Then, the Laplacian matrix $\LL := [l_{ij}]^N_{i,j=1} \in \RR^{N \times N}$ of $\GG$
can be further defined from $\mathcal{A}$ with $l_{ii} = \sum^N_{j=1} a_{ij}$ and, for $i \ne j$,\, $l_{ij} = -a_{ij}$. Moreover, $\LL$ is symmetric and positive semi-definite if and only if the graph $\GG$ is undirected \cite{Godsil2001}.

Given the switching signal $\sigma : [0, \infty) \to \mathcal{P} = \left\{1, 2,\ldots, n_0\right\}$
and $n_0$ graphs $\GG_{p} = (\mathcal{V}, \mathcal{E}_{p})$, $p = 1, 2, \ldots, n_0$,
each with the corresponding adjacency matrix denoted by $\mathcal{A}_{p}$
and the Laplacian matrix by $\LL_{p},\, p = 1, 2, \ldots, n_0$,
we call the time-varying graph $\GG_{\sigma(t)} = (\mathcal{V}, \mathcal{E}_{\sigma(t)})$ a switching graph,
and denote its adjacency matrix by $\mathcal{A}_{\sigma(t)}$
and its Laplacian matrix by $\LL_{\sigma(t)}$.
Finally, the graph $\GG := (\mathcal{V}, \mathcal{E})$ with $\mathcal{E} = \bigcup^q_{p=1} \mathcal{E}_{p}$
is called the union graph of the graphs $\GG_{p},\, p = 1, 2, \ldots, q$,
and is denoted by $\GG = \bigcup^q_{p=1} \GG_{p}$.

\section{Proof of Lemma \ref{lemma-ratio}}\label{Appendix-proof-lemma-ratio}

Before proving Lemma \ref{lemma-ratio},
we first present the following corollary of Lemma \ref{lemma-connect-graph}.

\begin{Corollary}\label{Corollary-connect-graph}
 Under Assumptions \ref{Ass-undirected} and \ref{Ass-connected},
there exists a number $0 < \delta < 1$ such that,
for any positive integer $\ell$,
\begin{equation}\label{eq-norm-products-le-1-forall-ell}
	\left\|\prod_{r=j_k}^{j_{k+\ell}-1} P_{\sigma(t_{r})}\right\| \leq \delta^{\ell}, \quad \forall\, k =0,1,2,\ldots.
\end{equation}
\end{Corollary}

\begin{Proof}
For any positive integer $\ell$, we have
\begin{equation*}
	\prod_{r=j_{k}}^{j_{k+\ell}-1} P_{\sigma(t_{r})} = \prod_{s=k}^{k+\ell-1} \left(\prod_{r=j_{s}}^{j_{s+1}-1} P_{\sigma(t_{r})} \right).
\end{equation*}
Then, by repeatedly using Lemma \ref{lemma-connect-graph},
we have
\begin{equation*}
	\left\|\prod_{r=j_{k}}^{j_{k+\ell}-1} P_{\sigma(t_{r})}\right\| \leq \prod_{s=k}^{k+\ell-1}\left\|\prod_{r=j_{s}}^{j_{s+1}-1} P_{\sigma(t_{r})}\right\| \leq \delta^\ell
\end{equation*}
which holds for any $k=0,1,2,\ldots$.
\end{Proof}

\medskip

Now we proceed to the proof of Lemma \ref{lemma-ratio}.

\begin{Proof}
First, under Assumptions \ref{Ass-undirected} and \ref{Ass-connected},
by using Corollary \ref{Corollary-connect-graph},
we can obtain from \eqref{eq-Psi} and \eqref{eq-expA} that
\begin{align}\label{eq-prod-Psi}
	&\quad\ \left\|\prod_{r=j_k}^{j_{k+\ell}-1}\Psi_{r}(\tau_r)\right\|  \notag \\
	& = \left\| \prod_{r=j_k}^{j_{k+\ell}-1} P_{\sigma(t_r)} \otimes \mathbf{e}^{A\tau_r} \right\| \notag\\
	& \le \left\| \prod_{r=j_k}^{j_{k+\ell}-1} P_{\sigma(t_r)}\right\|
	\left\|\mathbf{e}^{A (t_{j_{k+\ell}} - t_{j_k})}\right\| \notag\\
	& \le  \delta^{\ell} C_1 \mathbf{e}^{\lambda^* T_\mathrm{c} \ell}=C_{1} \left(\delta \mathbf{e}^{\lambda^* T_\mathrm{c} }\right)^{\ell}, \quad \forall\, k =0,1,2,\ldots.
\end{align}
Since $\lambda^*$ as selected in \eqref{eq-lambda-*-selection} that satisfies
$\lambda^* < - \frac{\ln \delta}{T_\mathrm{c}}$, we have $0<\delta \mathbf{e}^{\lambda^* T_\mathrm{c} }<1$.
Thus, by choosing the positive integer $\ell$ to be such that
\begin{align} \label{eq-l}
	\ell > - \frac{\ln(C_1)}{\ln \delta + \lambda^* T_\mathrm{c}}
\end{align}
we can obtain from \eqref{eq-prod-Psi} that
\begin{equation}\label{eq-C-0-ell-<1}
  \left\|\prod_{r=j_k}^{j_{k+\ell}-1}\Psi_{r}(\tau_r)\right\|
	\le C_{1} \left(\delta \mathbf{e}^{\lambda^* T_\mathrm{c} }\right)^{\ell} <1, \quad \forall\,  k =0,1,2,\ldots.
\end{equation}

Next, we derive an upper bound for the following term:
\begin{align*}
	\left\|\prod_{r=j_k}^{j_{k+\ell}-1}\Xi_r(\tau_r) - \prod_{r=j_k}^{j_{k+\ell}-1}\Psi_r(\tau_r) \right\|.
\end{align*}
By expanding the first product of matrices, we can obtain
\begin{align*}
	&\quad\, \prod_{r=j_k}^{j_{k+\ell}-1}\Xi_r(\tau_r) - \prod_{r=j_k}^{j_{k+\ell}-1}\Psi_r(\tau_r) \notag\\
	&=  \prod_{r=j_k}^{j_{k+\ell}-1}\left( \Phi_r(\tau_r) + \Psi_r(\tau_r)\right) -
	 \prod_{r=j_k}^{j_{k+\ell}-1} \Psi_r(\tau_r) \notag\\
	& =  \sum_{s=j_k}^{j_{k+\ell}-1}
	\Bigg( \prod_{r=s+1}^{j_{k+\ell}-1}\Xi_r(\tau_r) \cdot \Phi_s(\tau_s) \cdot \prod_{r=j_k}^{s-1} \Psi_r(\tau_r) \Bigg) \end{align*}
in which, the following convention is adopted:
\begin{align*}
\prod_{r=j_k}^{j_k-1} \Psi_r(\tau_r) = I_{Nn}, \quad
\prod_{r=j_{k+\ell}}^{j_{k+\ell}-1} \Xi_r(\tau_r) = I_{Nn}.
\end{align*}
Then, we have
\begin{align}\label{eq-product-Xi-product-Psi}
	&\quad\, \left\|\prod_{r=j_k}^{j_{k+\ell}-1}\Xi_r(\tau_r) - \prod_{r=j_k}^{j_{k+\ell}-1}\Psi_r(\tau_r) \right\| \notag\\
	& \le \sum_{s=j_k}^{j_{k+\ell}-1}
	\Bigg( \prod_{r=s+1}^{j_{k+\ell}-1}\left\|\Xi_r(\tau_r)\right\|  \cdot \left\|\Phi_s(\tau_s)\right\| \cdot \prod_{r=j_k}^{s-1} \left\|\Psi_r(\tau_r)\right\| \Bigg) \notag\\
	& \le \sum_{s=j_k}^{j_{k+\ell}-1}
	\Bigg( \prod_{r=s+1}^{j_{k+\ell}-1}\left(\left\|\Phi_r(\tau_r)\right\| + \left\|\Psi_r(\tau_r)\right\|\right) \cdot \left\|\Phi_s(\tau_s)\right\| \notag\\
	&\quad\quad\quad\quad\quad\quad\quad\quad\quad\quad \cdot \prod_{r=j_k}^{s-1} \left(\left\|\Phi_r(\tau_r)\right\| + \left\|\Psi_r(\tau_r)\right\|\right) \Bigg) \notag\\
&= \sum_{s=j_k}^{j_{k+\ell}-1}
\Bigg( \frac{\left\|\Phi_s(\tau_s)\right\|}{\left\|\Phi_s(\tau_s)\right\| + \left\|\Psi_s(\tau_s)\right\|}
 \notag\\
 &\quad\quad\quad\quad\quad\quad\quad\quad\quad
\cdot  \prod_{r=j_{k}}^{j_{k+\ell}-1}\left(\left\|\Phi_r(\tau_r)\right\| + \left\|\Psi_r(\tau_r)\right\|\right)\Bigg)\notag\\
	& \le \frac{(j_{k+\ell}-j_k)  \left\| \Phi_{s'}(\tau_{s'})\right\|}{\left\| \Psi_{s'}(\tau_{s'}) \right\| + \left\| \Phi_{s'}(\tau_{s'}) \right\|} \cdot \prod_{r=j_{k}}^{j_{k+\ell}-1}\left(\left\|\Phi_r(\tau_r)\right\| + \left\|\Psi_r(\tau_r)\right\|\right)
\end{align}
where $s'=\mathrm{argmax}_{s \in \left\{j_{k}, j_{k+1}, \ldots, j_{k+\ell}-1\right\}} \frac{\left\|\Phi_s(\tau_s)\right\|}{\left\|\Phi_s(\tau_s)\right\| + \left\|\Psi_s(\tau_s)\right\|} $.

From \eqref{eq-P-sigma-P-norm}, \eqref{eq-Psi}, and \eqref{eq-expA}, we have
\begin{equation*}
  \left\| \Psi_{r}(\tau_{r}) \right\| \le C_1 \mathbf{e}^{\lambda^* \tau_r}, \quad \forall\, r = j_k, j_k+1, \cdots, j_{k+\ell}-1.
\end{equation*}
Further, by Lemma \ref{lemma-Phi},
under Assumptions \ref{Ass-stabilizable} and \ref{Ass-undirected},
given any $\alpha>0$ and $t^* > 0$,
designing $K$ as in \eqref{eq-K} gives
\begin{equation*}
  \left\| \Phi_{r}(\tau_{r}) \right\| \le C_0(t^*) \mathbf{e}^{- \alpha \tau_r}, \  \forall\, r = j_k, j_k+1, \cdots, j_{k+\ell}-1.
\end{equation*}
Since it is clear that
$\mathbf{e}^{-\alpha t} \le 1 \le \mathbf{e}^{\lambda^* t},\, \forall\, t\ge0$,
we have
\begin{align*}
 &\quad\, \left\| \Psi_{r}(\tau_{r}) \right\| + \left\| \Phi_{r}(\tau_{r}) \right\| \notag\\
	&\leq \left(C_0(t^*) + C_1\right) \mathbf{e}^{\lambda^* \tau_r}, \quad \forall\, r = j_k, j_k+1, \cdots, j_{k+\ell}-1.
\end{align*}
Then, we can further obtain from \eqref{eq-product-Xi-product-Psi} that
\begin{align}\label{eq-lemma-Psi1}
	&\quad\, \left\|\prod_{r=j_k}^{j_{k+\ell}-1}\Xi_r(\tau_r) - \prod_{r=j_k}^{j_{k+\ell}-1}\Psi_r(\tau_r) \right\| \notag\\
	& \le  \left(j_{k+\ell}-j_k\right) \left(C_{0}(t^*)\mathbf{e}^{-\alpha \tau_{s'}}\right) \notag\\
	& \qquad\qquad\qquad \cdot\left((C_0(t^*) + C_1)^{j_{k+\ell} - j_k -1} \mathbf{e}^{\lambda^* (T_\mathrm{c} \ell - \tau_{s'})}\right) \notag\\
	& \le \frac{C_0(t^*)(j_{k+\ell}-j_k)}{C_0(t^*) + C_1} \notag\\
	& \qquad\qquad\qquad \cdot\left(C_0(t^*) + C_1\right)^{j_{k+\ell} - j_k}
\mathbf{e}^{\lambda^* T_\mathrm{c} \ell - (\alpha + \lambda^*)\tau}.
\end{align}
Since $j_{k+\ell}-j_k \le \left\lceil\frac{T_\mathrm{c}}{\tau}\right\rceil\ell,\, \forall\, k=0,1,2,\ldots$,
if we let
\begin{align*}
	C_3(t^*) = \frac{C_0(t^*)\left\lceil\frac{T_\mathrm{c}}{\tau}\right\rceil\ell}{C_0(t^*) + C_1}\left(\left(C_0(t^*) + C_1\right)^{\left\lceil\frac{T_\mathrm{c}}{\tau}\right\rceil} \mathbf{e}^{\lambda^* T_\mathrm{c}}\right)^{\ell}
\end{align*}
then it follows from  \eqref{eq-lemma-Psi1} that
\begin{align}\label{eq-lemma-Psi}
	\left\|\prod_{r=j_k}^{j_{k+1}-1}\Xi_r(\tau_r) - \prod_{r=j_k}^{j_{k+1}-1}\Psi_r(\tau_r) \right\| \leq C_3(t^*) \mathbf{e}^{-(\alpha + \lambda^*)\tau}.
\end{align}

Now, combining \eqref{eq-prod-Psi} and \eqref{eq-lemma-Psi} gives
\begin{align*}
	&\quad\, \left\|\prod_{r=j_k}^{j_{k+\ell}-1}\Xi_{r}(\tau_r)\right\| \notag\\
	&\le \left\|\prod_{r=j_k}^{j_{k+\ell}-1}\Psi_{r}(\tau_r)\right\| +
	\left\|\prod_{r=j_k}^{j_{k+1}-1}\Xi_r(\tau_r) - \prod_{r=j_k}^{j_{k+1}-1}\Psi_r(\tau_r) \right\| \notag\\
	&\le C_{1} \left(\delta \mathbf{e}^{\lambda^* T_\mathrm{c} }\right)^{\ell} + C_3(t^*) \mathbf{e}^{-(\alpha + \lambda^*)\tau},
\quad \forall \, k=0,1,2,\ldots.
\end{align*}
From \eqref{eq-l} and \eqref{eq-C-0-ell-<1},
we see that $\ell$ has been chosen such that $C_{1} \left(\delta \mathbf{e}^{\lambda^* T_\mathrm{c} }\right)^{\ell} <1$.
Thus, we can further choose the positive number $\alpha$ to be sufficiently large so that
\begin{align*}
	C_{1} \left(\delta \mathbf{e}^{\lambda^* T_\mathrm{c} }\right)^{\ell} + C_3(t^*) \mathbf{e}^{-(\alpha + \lambda^*)\tau} < 1.
\end{align*}
Consequently, there exist a positive integer $\ell$ satisfying \eqref{eq-l}
and a corresponding positive number $\alpha(\ell)$ defined as
\begin{align*}
	\alpha(\ell, t^*) :=  \frac{\ln C_3(t^*) - \ln \left(1 - C_{1} \left(\delta \mathbf{e}^{\lambda^* T_\mathrm{c} }\right)^{\ell}\right)}{\tau}- \lambda^*
\end{align*}
such that, if $\alpha > \alpha(\ell, t^*)$ and $K$ is designed as \eqref{eq-K},
then \eqref{eq-lemma-ratio} holds for 
\begin{align}\label{eq-relation}
\rho = C_1 \left(\delta \mathbf{e}^{\lambda^* T_\mathrm{c} }\right)^{\ell} + C_3(t^*) \mathbf{e}^{-(\alpha + \lambda^*)\tau}<1. 
\end{align}
The proof is thus complete.
\end{Proof}

\balance


\begin{thebibliography}{99}


\bibitem{CaiSuHuang2022}
H.~Cai, Y.~Su, and J.~Huang,
\textit{Cooperative Control of Multi-Agent Systems:
Distributed-Observer and Distributed-Internal-Model Approaches},
Switzerland: Springer, 2022.


\bibitem{Cheng2008}
D. Cheng, J. Wang, and X. Hu,
``An extension of LaSalle's invariance principle and its application to multi-agent consensus,''
\emph{IEEE Transactions on Automatic Control},
vol.~53, no.~7, pp.~1765--1770, 2008.



\bibitem{Godsil2001}
C. Godsil and G. F. Royle,
\emph{Algebraic Graph Theory},
Springer Science \& Business Media, 2001.




\bibitem{Hong2007}
Y. Hong, L. Gao, D. Cheng, and J. Hu,
``Lyapunov-based approach to multiagent systems with switching jointly connected interconnection,"
\textit{IEEE Transactions on Automatic Control},
vol.~52, no.~5, pp.~943--948, 2007.


\bibitem{Hu2007}
J. Hu and Y. Hong, ``Leader-following coordination of multi-agent systems with coupling time delays," \textit{Physica A: Statistical Mechanics and its Applications},
vol.~374, no.~2, pp.~853--863, 2007.



\bibitem{Jadbabaie2003}
A. Jadbabaie, J. Lin, and A. S. Morse,
``Coordination of groups of mobile autonomous agents using nearest neighbor rules,"
\textit{IEEE Transactions on Automatic Control},
vol.~48, no.~6, pp.~988--1001, 2003.








\bibitem{Lin2005}
Z. Lin, B. Francis, and Maggiore,
``Necessary and sufficient graphical conditions for formation control of unicycles,''
\emph{IEEE Transactions on Automatic Control},
vol.~50, no.~1, pp.~121--127, 2005.


 
\bibitem{LinWang2014}
Z. Lin, L. Wang, Z. Han, and M. Fu,
``Distributed formation control of multi-agent systems using complex Laplacian,'' 
\emph{IEEE Transactions on Automatic Control}, 
vol.~59, no.~7, pp.~1765--1777, 2014.



\bibitem{MaQin21}
Q. Ma, J. Qin, X. Yu, and L. Wang,
``On necessary and sufficient conditions for exponential consensus in dynamic networks via uniform complete observability theory,"
\textit{IEEE Transactions on Automatic Control}, vol. 66, no. 10, pp. 4975--4981, 2021.


%





\bibitem{Nedic09}
 A. Nedi\'{c} and A. Ozdaglar,
 ``Distributed subgradient methods for multi-agent optimization,''
 \emph{IEEE Transactions on Automatic Control},
 vol.~54, no.~1, pp.~48--61, 2009.


 
\bibitem{Nedic15}
 A. Nedi\'{c} and A. Olshevsky,
 ``Distributed optimization over time-varying directed graphs,'' 
 \emph{IEEE Transactions on Automatic Control},
 vol.~60, no.~3, pp.~601--615, 2015.

 



\bibitem{Mohar-1991GC}
B. Mohar,
``Eigenvalues, diameter, and mean distance in graphs,"
\textit{Graphs Combinatories},
vol.~7, no.~1, pp.~53--64, 1991.


\bibitem{Ni2010}
W. Ni and D. Cheng,
``Leader-following consensus of multi-agent systems under fixed and switching topologies,"
\textit{Systems \& Control Letters},
vol.~59, no.~3, pp.~209--217, 2010.


 
\bibitem{Olfati-SaberMurray04}
R. Olfati-Saber and R. M. Murray,
``Consensus problems in networks of agents with switching topology and time-delays,''
\textit{IEEE Transactions on Automatic Control},
vol.~49, no.~9, pp.~1520--1533, 2004.



\bibitem{Qin2014}
J. Qin and C. Yu,
``Exponential consensus of general linear multi-agent systems under directed dynamic topology,''
\emph{Automatica},
vol.~50, no.~9, pp.~2327--2333, 2014.



\bibitem{Ren2005}
 W. Ren and R. W. Beard,
 ``Consensus seeking in multiagent systems under dynamically changing interaction topologies,"
 \textit{IEEE Transactions on Automatic Control},
 vol.~50, no.~5, pp.~655--661, 2005.
 
 
 
\bibitem{Ren2007}
W. Ren and E. Atkins,
``Distributed multi-vehicle coordinated control via local information exchange,'' 
\emph{International Journal of Robust and Nonlinear Control},
vol.~17, pp.~1002--1033, 2007.



\bibitem{Ren2008}
W. Ren,
``On consensus algorithms for double-integrator dynamics,"
\textit{IEEE Transactions on Automatic Control},
vol.~53, no.~6, pp.~1503--1509, 2008.




\bibitem{SuHuang12}
Y. Su and J. Huang,
``Cooperative output regulation of linear multi-agent systems,''
\emph{IEEE Transactions on Automatic Control},
vol.~57, no.~4, pp.~1062--1066, 2012.



 
\bibitem{SuHuang12Cyber}
Y. Su and J. Huang, 
``Cooperative output regulation with application to multi-agent consensus under switching network,'' 
\emph{IEEE Transactions on Systems, Man, and Cybernetics, Part B (Cybernetics)}, 
vol.~42, no.~3, pp.~864--875, 2012.


\bibitem{SuHuang-2012TAC}
Y. Su and J. Huang,
``Stability of a class of linear switching systems with applications to two consensus problems,"
\textit{IEEE Transactions on Automatic Control},
vol.~57, no.~6, pp.~1420--1430, 2012.


 \bibitem{SuLee22}
 Y. Su and T. C. Lee,
 ``Output feedback synthesis of multiagent systems with jointly connected switching networks:
 A separation principle approach,"
 \textit{IEEE Transactions on Automatic Control},
 vol.~67, no.~2, pp.~941--948, 2022.
 

 \bibitem{Tuna08}
 S. E. Tuna,
 ``LQR-based coupling gain for synchronization of linear systems,'' 
 arXiv preprint, arXiv:0801.3390, 2008.



\bibitem{WangMorse18}
L. Wang and A. S. Morse,
``A distributed observer for a time-invariant linear system,''
\emph{IEEE Transactions on Automatic Control},
vol.~63, no.~7, pp.~2123--2130, 2018.



\bibitem{WangZhuFeng2019TAC}
X. Wang, J. Zhu, and J. Feng,
``A new characteristic of switching topology and synchronization of linear multiagent systems,"
\textit{IEEE Transactions on Automatic Control},
vol.~64, no.~7, pp.~2697--2711, 2019.







\bibitem{YeHu17}
M. Ye and G. Hu,
``Distributed Nash equilibrium seeking by a consensus based approach,''
\emph{IEEE Transactions on Automatic Control},
vol.~62, no.~9, pp.~4811--4818, 2017.


 
\bibitem{YeHu18}
M. Ye, G. Hu, and F. L. Lewis,
``Nash equilibrium seeking for $N$-coalition noncooperative games,''
\emph{Automatica},
vol.~95, pp.~266--272, 2018.
 


 
\bibitem{ZhangLu24}
L. Zhang, M. Lu, F. Deng, and J. Chen, 
``Distributed state estimation under jointly connected switching networks: Continuous-time linear systems and discrete-time linear systems,'' 
\emph{IEEE Transactions on Automatic Control}, 
vol.~69, no.~2, pp.~1104--1111, 2024.
 




\end{thebibliography}
\end{document}